\documentclass[12pt]{article}

\usepackage{oldlfont}
\usepackage{amsfonts}
\usepackage{amsmath}
\usepackage{amssymb}
\usepackage{color}
\usepackage{float}
\usepackage{graphicx}
\usepackage{eepic}
\usepackage{a4wide}

\usepackage[T1]{fontenc}
\usepackage[latin1]{inputenc}

\newcommand{\be}{\begin{equation}}
\newcommand{\ee}{\end{equation}}
\newcommand{\bea}{\begin{eqnarray}}
\newcommand{\eea}{\end{eqnarray}}
\newcommand{\beas}{\begin{eqnarray*}}
\newcommand{\eeas}{\end{eqnarray*}}
\newcommand{\ba}{\begin{array}}
\newcommand{\ea}{\end{array}}

\newtheorem{theorem}{Theorem} [section]
\newtheorem{lemma}{Lemma} [section]
\newtheorem{corollary}{Corollary} [section]
\newtheorem{prop}{Proposition} [section]
\newtheorem{remark}{Remark} [section]

\newtheorem{counterex}{Counterexample} [section]
\newenvironment{proof}{\noindent\textbf{Proof.}\ }
              {\nopagebreak\hbox{ }\hfill$\Box$\bigskip}
\newcommand{\qed}{\nopagebreak\hbox{ }\hfill$\Box$\bigskip}

\newcommand{\grad}{\nabla}
\renewcommand{\div}{{\mathrm div}}

\newcommand{\bcurl}{{\mathbf curl}}

\newcommand{\divg}{{\mathrm div}_{\Gamma}}

\newcommand{\eps}{\varepsilon}
\newcommand{\G}{\Gamma}

\newcommand{\tG}{\widetilde\Gamma}
\newcommand{\tbn}{{\tilde{\mathbf n}}}
\newcommand{\tH}{\tilde H}

\newcommand{\boldf}{{\mathbf f}}

\newcommand{\bn}{{\mathbf n}}
\newcommand{\bq}{{\mathbf q}}
\newcommand{\bu}{{\mathbf u}}
\newcommand{\bv}{{\mathbf v}}
\newcommand{\bw}{{\mathbf w}}
\newcommand{\bx}{{\mathbf x}}
\newcommand{\boldy}{{\mathbf y}}
\newcommand{\bz}{{\mathbf z}}

\newcommand{\bH}{{\mathbf H}}
\newcommand{\bL}{{\mathbf L}}

\newcommand{\bV}{{\mathbf V}}
\newcommand{\bW}{{\mathbf W}}
\newcommand{\bX}{{\mathbf X}}

\newcommand{\bnu}{\hbox{\mathversion{bold}$\nu$}}

\newcommand{\bPhi}{\hbox{\mathversion{bold}$\Phi$}}
\newcommand{\bPsi}{\hbox{\mathversion{bold}$\Psi$}}

\newcommand{\CM}{{\cal M}}

\newcommand{\CQ}{{\cal Q}}
\newcommand{\CT}{{\cal T}}

\newcommand{\bCR}{\boldsymbol {\cal R}}

\newcommand{\bCT}{\boldsymbol {\cal T}}

\newcommand{\sfE}{\operatorname{\mathsf{E}}}
\newcommand{\sfL}{\operatorname{\mathsf{L}}}

\newcommand{\sfR}{\operatorname{\mathsf{R}}}

\newcommand{\stack}[2]{\mathrel{\mathop{#2}\limits_{\scriptstyle #1}}}
\newcommand{\bzero}{{\mathbf 0}}
\newcommand{\<}  {\langle}
\renewcommand{\>}{\rangle}
\newcommand{\field}[1]{\mathbb{#1}}

\definecolor{otherblue}{rgb}{0,0.3,0.6}
\definecolor{gr}{rgb}   {0.,   0.69,   0.23 }
\def\rev#1{\textcolor{blue}{#1}}     
\def\rblf#1{\textcolor{otherblue}{#1}}

\def\rkserge#1{\textcolor{gr}{#1}}
\def\rev#1{\textcolor{black}{#1}}
\def\rblf#1{\textcolor{black}{#1}}

\def\rkserge#1{\textcolor{black}{#1}}

\title{The BEM with graded meshes for the electric field integral equation
  on polyhedral surfaces}

\author{A. Bespalov\thanks{School of Mathematics, University of Birmingham,
  Edgbaston, Birmingham B15 2TT, UK ({\tt a.bespalov@bham.ac.uk}).}
  \and
  S. Nicaise\thanks{Laboratoire de Math\'ematiques et ses Applications de Valenciennes, FR CNRS 2956,
  Institut des Sciences et Techniques de Valenciennes,
  Universit\'e de Valenciennes et du Hainaut-Cambr\'esis,
  Le Mont Houy, 59313
  Valenciennes Cedex 9, France ({\tt serge.nicaise@univ-valenciennes.fr}).}
}

\begin{document}

\date{}
\maketitle

\begin{abstract}
We consider the variational formulation
of the electric field integral equation on a Lipschitz polyhedral surface $\G$.
\rev{We study} the Galerkin boundary element \rev{discretisations}
\rev{based on the lowest-order Raviart-Thomas surface elements on a sequence of
anisotropic meshes algebraically graded towards} the edges of $\G$.
We establish quasi-optimal convergence of Galerkin solutions
\rev{under a mild restriction on the strength of grading}.
The key ingredient of our convergence analysis
are new \rev{componentwise} stability properties
of the Raviart-Thomas \rev{interpolant} 
on anisotropic elements.
\end{abstract}




\section{Introduction} \label{sec_intro}
\setcounter{equation}{0}

In this paper, we study the Galerkin boundary element method (BEM) 
on graded meshes
for numerical solution of the electric field integral equation (EFIE)
on a Lipschitz polyhedral surface $\G$ in~${\field{R}}^{3}$
(i.e., $\Gamma=\partial\Omega$, where $\Omega\subset{\field{R}}^{3}$ is a Lipschitz polyhedron).
The EFIE models the scattering of time-harmonic electromagnetic waves at a perfect conductor,
and the Galerkin BEM is widely used in engineering practice for simulation of this physical phenomenon.

The Galerkin BEM considered in this paper employs $\divg$-conforming lowest-order
Raviart-Thomas surface elements to discretise the variational formulation of the EFIE
(known as Rumsey's principle).
This approach is referred to as the natural BEM for the EFIE
(there exist other approaches, e.g., based on a stable mixed reformulation of Rumsey's principle,
see~\cite{BuffaCS_02_BEM}).
Non-coercivity of the bilinear form in Rumsey's principle
(due to the infinite-dimensional kernel of $\divg$, cf.~(\ref{Rumseys}))
significantly complicates the convergence analysis of Galerkin schemes.
This problem can be overcome by using appropriate decompositions of vector fields
in order to isolate the kernel of $\divg$
(we refer to discussion in~\cite[Section~3]{BuffaH_03_GBE},
to an abstract theory in~\cite{BuffaC_03_EFI, Buffa_05_RDS}, and
we outline available techniques for constructing such decompositions
in Section~\ref{sec_decomp}).
These ideas \rev{have led} to major advances in the convergence analysis and a priori error analysis
of the BEM for the EFIE on (open and closed) Lipschitz surfaces,
see~\cite{HiptmairS_02_NBE, BuffaCS_02_BEM, BuffaC_03_EFI, BuffaHvPS_03_BEM, BuffaH_03_GBE}
for the $h$-version of the BEM
and~\cite{BespalovH_10_NpB, BespalovHH_10_Chp, BespalovH_10_hpA, BespalovH_12_Nhp}
for high-order methods ($p$- and $hp$-BEM).
All these results, however, assume shape-regularity of the underlying meshes on $\G$.

It is well-known that convergence rates of the $h$-BEM with quasi-uniform and shape-regular meshes
are bounded by the poor regularity of solutions to the EFIE on non-smooth surfaces.
For example, on a closed polyhedral surface $\G = \partial\Omega$,
the solution may be only $\bH^{\eps}(\G)$-regular
(with a small $\eps>0$ in the case of non-convex polyhedron $\Omega$,
cf.~\cite[Section~4.4.2]{CostabelD_00_SEF}),
and convergence rate of the $h$-BEM is only $\frac 12 + \eps$ in this case,
\rev{whereas in the case of smooth solutions the lowest-order $h$-BEM converges with the optimal rate of $\frac 32$}
(see~\cite[Theorem~8.2]{HiptmairS_02_NBE} and~\cite[Theorem~2.2]{BespalovH_10_hpA}).
Taking the cue from the $h$-BEM results for the Laplacian (see~\cite{vonPetersdorffS_90_RMB,vP}),
\rev{we expect that} an optimal convergence rate of the $h$-BEM
\rev{for the EFIE can be recovered on the non-smooth surface $\G$, if one employs}
the meshes that are appropriately graded towards the edges of~$\G$.
These meshes contain highly anisotropic elements along the edges \rev{of~$\G$,
and none of the results mentioned above is applicable in this case.
Moreover, to the best of our knowledge, the quasi-optimality of the Galerkin $h$-BEM with graded meshes
for the EFIE has not been studied in the literature,}
and with this paper we fill this theoretical gap.

In the next section, we introduce necessary notation and formulate the EFIE in its variational form.
In Section~\ref{sec_bem}, we construct graded meshes on $\G$,
introduce the boundary element space, and formulate the main result of the paper---
Theorem~\ref{thm_converge}---that establishes quasi-optimal convergence of Galerkin solutions on graded meshes.
\rev{The proof of Theorem~\ref{thm_converge} follows the approach suggested
in~\cite{BuffaC_03_EFI, BuffaHvPS_03_BEM}, summarised in~\cite[Section~9.1]{BuffaH_03_GBE},
and extended to a general class of operators in~\cite[Section~3]{Buffa_05_RDS}.
At the heart of this approach is the decomposition technique described in Section~\ref{sec_decomp}.}
Section~\ref{sec_rt_stab} is instrumental in the construction of the corresponding discrete decomposition:
here we establish new stability properties of the Raviart-Thomas interpolant \rev{of} low-regular
vector fields on anisotropic elements.
In Section~\ref{sec_discrete_decomp}, we introduce the discrete decomposition
and \rev{complete the proof of} Theorem~\ref{thm_converge}.
An essential ingredient here is the projection operator $\CQ_h$
with enhanced approximation properties (see Proposition~\ref{prop_Qh-projector}).
\rev{The proof of Proposition~\ref{prop_Qh-projector} is given in Section~\ref{sec_appendix}.}

\renewcommand{\tG}{F}

\section{The electric field integral equation} \label{sec_efie}
\setcounter{equation}{0}

The variational formulation of the EFIE is posed on the Hilbert space
\[
  \bX = \bH^{-1/2}(\divg,\G) :=
        \{\bu \in \bH^{-1/2}_{\|}(\G);\; \divg\,\bu \in H^{-1/2}(\G)\}.
\]
Here, $\divg$ denotes the surface divergence operator,
$\bH^{-1/2}_{\|}(\G)$ is the dual space of $\bH^{1/2}_{\|}(\G)$
(the tangential trace space of $\bH^1(\Omega)$ on $\G$, see~\cite{BuffaC_01_TFI, BuffaCS_02_THL}),
and $H^{-1/2}(\G)$ is the dual space of $H^{1/2}(\G)$.
The space $\bX$ is equipped with its graph norm $\|\cdot\|_{\bX}$.
We refer to~\cite{BuffaC_01_TFI, BuffaC_01_TII, BuffaCS_02_THL, BuffaH_03_GBE}
for definitions and properties of $\bH^{-1/2}(\divg,\Gamma)$ and other involved trace spaces.
We also recall from \cite{BuffaC_01_TFI, BuffaCS_02_THL} that $\bX$ is the natural tangential trace space of
$\bH(\bcurl,\Omega)$.

In the present article, we use the same notation as in~\cite{BespalovH_10_hpA},
where we recalled definitions of the full range of Sobolev spaces and differential operators
needed for convergence analysis of the BEM for the EFIE (see Section~{3.1} therein).
In particular, we use a traditional notation for the Sobolev spaces (of scalar functions)
$H^s$, $\tH^s$ ($s \in [-1,1]$), $H^s_0$ ($s \in (0,1]$)
and their norms on Lipschitz domains and surfaces (see~\cite{LionsMagenes, McLean_00_SES}).
The norm and inner product in $L^2(D) = H^0(D)$ on a domain or surface~$D$ will be denoted by
$\|\cdot\|_{0,D}$ and $(\cdot,\cdot)_{0,D}$, respectively.
The notation $(\cdot,\cdot)_{0,D}$ will be used also for appropriate
duality pairings extending the $L^2(D)$-pairing for functions on~$D$.

For vector fields we will use boldface symbols (e.g., $\bu = (u_1,u_2)$), and
the spaces (or sets) of vector fields are also denoted in boldface
(e.g., $\bH^s(D) = (H^s(D))^2$ with $D \subset {\field{R}}^2$).
The norms and inner products in these spaces are defined componentwise.
The notation for the Sobolev spaces of tangential vector fields on $\G$
follows~\cite{BuffaC_01_TFI, BuffaC_01_TII, BuffaCS_02_THL}.
In particular, $\bL^2_{\rm t}(\G)$ denotes the space of two-dimensional, tangential,
square integrable vector fields on $\G$. The norm and inner product in this space
will be denoted by $\|\cdot\|_{0,\G}$ and $(\cdot,\cdot)_{0,\G}$, respectively,
and we will also use $(\cdot,\cdot)_{0,\G}$ for appropriate
duality pairings extending the $\bL^2_{\rm t}(\G)$-pairing for tangential vector fields on $\G$.
The similarity of this notation with the one for scalar functions
should not lead to any confusion, as the meaning will always be clear from the context .

For a fixed wave number $k>0$ and for a given source functional $\boldf \in \bX'$,
the variational formulation for the EFIE reads as:
{\em find a complex tangential field $\bu\in\bX$ such that}
\be \label{Rumseys}
    a(\bu,\bv) := \<\Psi_k\divg\,\bu,\divg\,\bv\> - k^2 \<\bPsi_k\bu, \bv\> =
    \<\boldf, \bv\> \quad\forall \bv \in \bX.
\ee
Here, $\Psi_k$ (resp., $\bPsi_k$) denote
the scalar (resp., vectorial) single layer boundary integral operator on $\Gamma$
for the Helmholtz operator $-\Delta - k^{2}$, see~\cite[Section~4.1]{BuffaCS_02_BEM}
(resp.,~\cite[Section~5]{BuffaH_03_GBE}).

To ensure the uniqueness of the solution to (\ref{Rumseys})
we always assume that $k^2$ is not an electrical eigenvalue of the interior problem in~$\Omega$.

\renewcommand{\hat}{\widehat}
\newcommand{\RT}{\bCR\bCT}

\section{Galerkin BEM on graded meshes. \rev{The main result}.} \label{sec_bem}
\setcounter{equation}{0}

For approximate solution of (\ref{Rumseys}) we apply \rev{the} natural BEM
based on Galerkin discretisations with lowest-order Raviart-Thomas spaces
on graded meshes.

First, let us describe the construction of graded meshes on individual faces of $\G$.
Here, we follow~\cite[Section~3]{vonPetersdorffS_90_RMB}.
For simplicity, we can assume that all faces of $\G$ are triangles.
On general polygonal faces the construction is similar, or one can first
subdivide the polygon into triangles.
On a triangular face $\tG \subset \G$, we first draw three lines through
the centroid and parallel to the sides of $\tG$.
This makes $\tG$ divided into three parallelograms and three triangles (see Figure~\ref{fig_1}).
Each of the three parallelograms can be mapped onto the unit square $\widehat Q = (0,1)^2$
by a linear transformation such that the vertex $(0,0)$ of $\hat Q$ is the image of a vertex of $\tG$.
Analogously, each of the three sub-triangles can be mapped onto the unit triangle
$\widehat T = \{\bx = (x_1,x_2);\; 0 < x_1 < 1,\ 0 < x_2 < x_1\} \subset \hat Q$
such that the vertex~$(1,1)$ of~$\hat T$ is the image of the centroid of~$\tG$.
Next, the graded mesh on $\hat Q$ (and hence on~$\hat T$) is generated by the lines
\[
  x_1 = \left(\frac iN\right)^\beta,\quad
  x_2 = \left(\frac jN\right)^\beta,\quad
  i,j = 0,1,\ldots,N.
\]
Here, $\beta \ge 1$ is the grading parameter, and $N \ge 1$ corresponds to the level of refinement.
Mapping each cell of these meshes back onto the face $\tG$, we obtain a graded mesh of triangles
and parallelograms on $\tG$ (see Figure~\ref{fig_1}).
Note that the diameter of the largest element of this mesh is proportional to $\beta N^{-1}$.
Hence, $h = 1/N$ defines the mesh parameter, and we will denote by
$\CT = \{\Delta_h^{\beta}\}$ a family of graded meshes
$\Delta_h^{\beta} = \{K;\; \cup \bar K = \bar\G\}$
generated on $\G$ by following the procedure described above.

\begin{figure}[!htb]
\begin{center}
\includegraphics[width=1\textwidth]{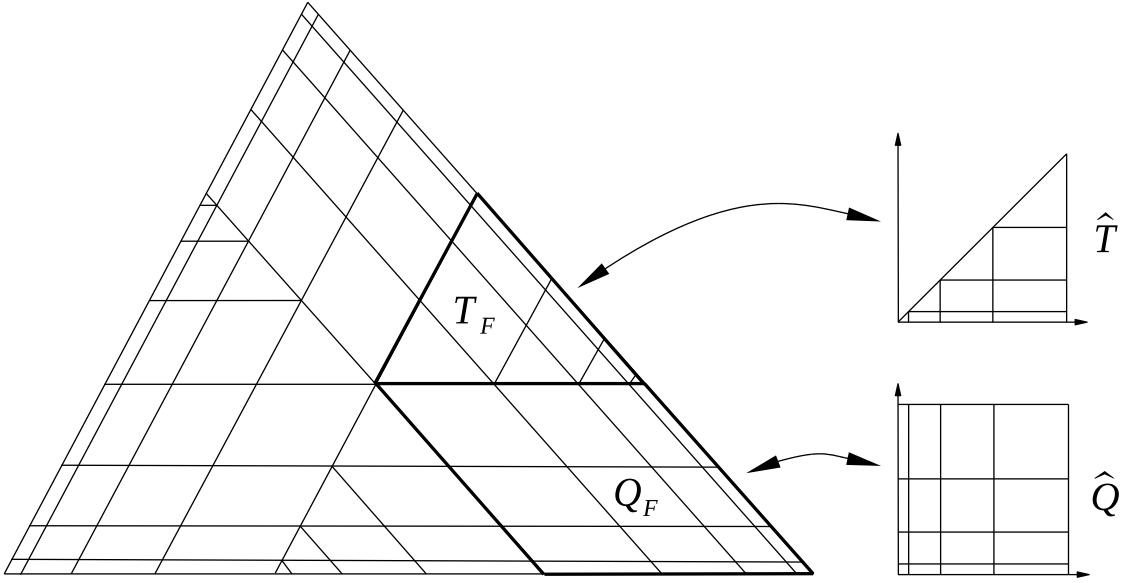}
\addtolength{\abovecaptionskip}{-2em}
\end{center}
\caption{Graded mesh on the triangular face $\tG \subset \G$.
The triangular (resp., parallelogram) block of elements $T_F$ (resp.,~$Q_F$)
is the image of the graded mesh on the unit triangle $\hat T$ (resp., the unit square~$\hat Q$).}
\label{fig_1}
\addtolength{\abovecaptionskip}{1em}
\end{figure}

Let us now introduce the boundary element space $\bX_h$.
It is known that Raviart-Thomas surface elements provide an affine equivalent family of
$\divg$-conforming finite elements under the Piola transformation,
see~\cite[Section~III.3]{BrezziF_91_MHF}.
We will write $\RT_0(K)$ for the local lowest-order Raviart-Thomas space
on a generic (triangular or quadrilateral) element $K$, and we denote by
$\bX_h = \RT_0(\Delta_h^\beta)$ the corresponding space
of $\divg$-conforming boundary elements over the graded mesh $\Delta_h^\beta$.

The following theorem states the unique solvability and quasi-optimal convergence
of the Galerkin BEM on graded meshes for the EFIE.

\begin{theorem} \label{thm_converge}
There exists $h_0 < 1$ such that for any $\boldf \in \bX'$ and for any
graded mesh $\Delta_h^\beta$ with $h \le h_0$ and $\beta \in [1,3)$,
the Galerkin boundary element discretisation of {\rm (\ref{Rumseys})}
admits a unique solution $\bu_h \in \bX_h$ and the $h$-version of
the Galerkin BEM on graded meshes $\Delta_h^\beta$ converges quasi-optimally, i.e.,
\be \label{quasi-optimal}
    \|\bu - \bu_h\|_{\bX} \le C \inf_{\bv \in \bX_h} \|\bu - \bv\|_{\bX},
\ee
where the constant $C$ may depend only on the geometry of $\G$ and
the grading parameter~$\beta$.
\end{theorem}

The proof of Theorem~\ref{thm_converge} relies on an abstract theory
\rev{for analysing convergence of}
Galerkin discretisations \rev{for} non-coercive variational problems like (\ref{Rumseys}).
This theory was developed in~\cite{BuffaC_03_EFI}, \cite{BuffaHvPS_03_BEM},
\rev{\cite[Section~9.1]{BuffaH_03_GBE}, and in}~\cite[Section~3]{Buffa_05_RDS}.
In particular, it follows from the latter article that in order to prove
Theorem~\ref{thm_converge} we need to establish the following properties:
\begin{itemize}
\item[\textbf{(A)}]
the existence of a stable direct decomposition $\bX=\bV\oplus\bW$ such that
$a|_{\bV \times \bV}$ and $\rev{-}a|_{\bW \times \bW}$ are both $\bX$-coercive, and
$a|_{\bV \times \bW}$ and $a|_{\bW \times \bV}$ are both compact;
\item[\textbf{(B)}]
the existence of the corresponding discrete decomposition $\bX_{h} = \bV_{h} + \bW_{h}$,
$\bW_{h} \subset \bW$, that is uniformly stable with respect to the mesh parameter $h$;
\item[\textbf{(C)}]
the gap property
\be \label{gap}
    \sup\limits_{\bv_{h} \in \bV_{h}}  \inf\limits_{\bv \in \bV}
    \frac{\|\bv-\bv_{h}\|_{\bX}}{\|\bv_{h}\|_{X}} \leq \eps(h)\quad
    \hbox{with \ $\eps(h) \to 0$ \ as \ $h \to 0$.}
\ee
\end{itemize}
We will prove Theorem~\ref{thm_converge} by verifying these properties
in Sections~\ref{sec_decomp} and~\ref{sec_discrete_decomp} below.

\begin{remark} \label{rem_open_surface_1}
Theorem~{\rm \ref{thm_converge}} remains valid if $\G$ 
is a piecewise plane orientable open surface
(see~{\rm \cite{BuffaC_03_EFI}} for the problem formulation and the underlying tangential trace spaces
in this case). The proof repeats the arguments
in Sections~{\rm \ref{sec_decomp}} and~{\rm \ref{sec_discrete_decomp}} below
by using a specific construction of the decomposition $\bX = \bV \oplus \bW$
as described in~{\rm \cite[Section~3]{BespalovHH_10_Chp}}.
\end{remark}

\begin{remark} \label{rem_h-rate}
\rev{If some information about the regularity of the solution $\bu$ to {\rm (\ref{Rumseys})} is available,
then convergence result of Theorem~{\rm \ref{thm_converge}} translates into an a priori error estimate
in the natural $\bX$-norm.
For scattering problems with sufficiently smooth source functional $\boldf$
(e.g., with $\boldf$ representing the excitation by an incident plane wave),
the regularity of the solution depends only on the geometry of $\G$.
In particular, nonsmoothness of $\G$ leads to singularities in the solution of the EFIE,
severely affecting convergence rates of the $h$-BEM on shape-regular meshes.
However, similar to the case of the Laplacian in~{\rm \cite{vonPetersdorffS_90_RMB,vP}},
by employing the graded meshes with sufficiently large grading parameter $\beta$
(depending on the strength of singularities in $\bu$)
one may hope to recover
the optimal convergence rate (i.e., the rate of the $h$-BEM on quasi-uniform meshes
in the case of a smooth solution).
The main question here is whether the restriction on the grading parameter $\beta$
that guarantees quasi-optimality of the Galerkin BEM in Theorem~{\rm \ref{thm_converge}},
is sufficient for recovering this optimal convergence.
We will address this issue in the forthcoming article~{\rm \cite{BespalovN_AEA}}.
}
\end{remark}

Throughout the paper, $C$, $C_1$, {\em etc.} denote generic positive constants
that are independent of the mesh parameter $h$ and involved functions but
may depend on the geometry of $\G$ and the grading parameter~$\beta$.
We will also write
$a \lesssim b$
and $A \simeq B$, which means the existence
of generic positive constants $C$, $C_1$, $C_2$ such that
$a \le C b$ and $C_1 B \le A \le C_2 B$, respectively.

\section{Decomposition technique} \label{sec_decomp}
\setcounter{equation}{0}

Let us address property {\bf (A)} in Section~\ref{sec_bem}.
One way to obtain a suitable decomposition is to employ the $\bL^2_{\rm t}(\G)$-orthogonal
Hodge decomposition of $\bX$, cf.~\cite{BuffaC_01_TII}.
In the context of the $h$-BEM on shape-regular meshes, this idea was successfully exploited
in~\cite{BuffaCS_02_BEM, HiptmairS_02_NBE, BuffaHvPS_03_BEM, BuffaC_03_EFI}.
However, in the case of non-smooth surfaces, the regularity of surface gradients
in the $\bV$-component of the decomposition may be poor, and this causes substantial
technical difficulties in the analysis of the $p$- and $hp$-BEM.
For the $p$-BEM on plane open screens, a modification of the above strategy was suggested
in~\cite{BespalovH_10_NpB}, where we consistently used the $\tilde\bH^{-1/2}$-inner product
and proved $\tilde\bH^{-1/2}$-orthogonality of the Hodge decomposition.
Unfortunately, these ideas do not generalise immediately to closed polyhedral surfaces
or piecewise plane open screens, neither to the $hp$-BEM with quasi-uniform meshes,
cf.~\cite{BespalovH_Chp}. When attempting to use the Hodge decomposition for convergence
analysis of the $h$-BEM on graded meshes, poor regularity of vector fields in the $\bV$-component
leads to severe restrictions on the grading parameter $\beta$.

An alternative technique employs a regularising projection
$\sfR: \bX \to \bX$ to construct a decomposition of $\bX$
with enhanced regularity of the $\bV$-component
(see~\cite{Hiptmair_03_CFE}, \cite[Section~3]{BuffaH_03_GBE}, and~\cite[Section~4.3.1]{Buffa_05_RDS}).
The projection $\sfR$ is defined by employing the $\bH^1(\Omega)$-regular vector potentials
from the following lemma.

\begin{lemma} \label{lm_L-map}
  {\rm \cite[Section~3]{AmroucheBDG_98_VPT}}
  For any bounded Lipschitz domain $\Omega \subset {\field{R}}^{3}$ there exists a continuous mapping
  $\sfL:\bcurl\, \bH(\bcurl,\Omega) \to \bH^{1}(\Omega)$
  such that 
  $\bcurl \sfL \bPhi = \bPhi$ for all $\bPhi \in \bcurl\, \bH(\bcurl,\Omega)$.
\end{lemma}

Let $\bu \in \bX$.
The regularising projection $\sfR: \bX \mapsto \bX$ is defined as follows:
\[
  \sfR\bu := \big((\sfL \bPhi) \times \bnu\big)|_{\G},
\]
where $\sfL$ is from Lemma~\ref{lm_L-map},
$\bnu$ denotes the unit outward normal to $\Omega$,
$\bPhi := \grad w$, and
$w \in H^1(\Omega)$ is the solution to the problem
\beas
     -\Delta w = 0 & & \hbox{in $\Omega$},
     \\[2pt]
     \grad w \cdot \bnu = \divg \bu & & \hbox{on $\G$}.
\eeas
The fact that $\int_{\Sigma} \divg\bu\, dS = 0$ for each connected component $\Sigma$ of $\G$
guarantees $\bPhi \in \bcurl\, \bH(\bcurl, \Omega)$, and Lemma~\ref{lm_L-map} can be applied.
By using elliptic lifting theorems, trace theorems, and the continuity of $\sfL$, we conclude:
\be \label{R-bound}
    \exists C = C(\G) > 0\ \ \hbox{such that}\ \
    \|\sfR \bu\|_{\bH^{1/2}_{\perp}(\G)} \le
    C\, \|\divg\bu\|_{H^{-1/2}(\G)}\ \ 
    \forall\,\bu \in \bX,
\ee
where $\bH^{1/2}_{\perp}(\G) \subset \bX$ is the rotated tangential trace space
of $\bH^1(\Omega)$ on $\G = \partial\Omega$, see~\cite{BuffaC_01_TFI, BuffaCS_02_THL}.

By construction of $\sfR$, we have on $\G$
\be \label{div-property}
    \divg \sfR \bu = \divg \bu\ \ \hbox{for any $\bu \in \bX$},
\ee
that is $\sfR^2 = \sfR$.

Now we can define the decomposition
\be \label{decomp}
    \bX = \bV \oplus \bW
    \quad \hbox{with}\quad
    \bV := \sfR(\bX) \subset \bH^{1/2}_{\perp}(\G)\quad \hbox{and}\quad
    \bW := ({\rm Id} - \sfR)(\bX).
\ee
\rev{Decomposition (\ref{decomp}) was used in~\cite{BespalovHH_10_Chp}}
to prove unique solvability and quasi-optimal convergence of the $hp$-BEM with locally variable
polynomial degrees on shape-regular meshes.
As we will see \rev{in this paper}, the same \rev{decomposition} technique
can be used effectively in the analysis of the $h$-BEM on graded meshes.

By (\ref{div-property}) we conclude that $\bW$ comprises $\divg$-free vector fields.
Stability of decomposition (\ref{decomp}) follows from inequality (\ref{R-bound}) and
the continuous embedding $\bH^{1/2}_{\perp}(\G) \hookrightarrow \bX$.
Furthermore, the embedding $\bV \hookrightarrow \bL^2_{\rm t}(\G)$ is compact
by (\ref{R-bound}) and Rellich's theorem.
Thus, thanks to the $H^{-1/2}(\G)$-coercivity (resp., $\bH_{\perp}^{-1/2}(\G)$-coercivity)
of $\Psi_k$ (resp. $\bPsi_k$) (see~\cite[Lemmas~8, 7]{BuffaH_03_GBE}),
the $\bX$-coercivity of $a|_{\bV \times \bV}$ and $a|_{\bW \times \bW}$
is proved by the same arguments as in~\cite[proof of Theorem~3.4]{BuffaC_03_EFI}.
The compactness of $a|_{\bV \times \bW}$ and $a|_{\bW \times \bV}$
is due to the continuity of $\bPsi_k: \bH_{\perp}^{-1/2}(\G) \to \bH_{\perp}^{1/2}(\G)$
and the compactness of the embedding $\bV \hookrightarrow \bH_{\perp}^{-1/2}(\G)$
(see~\cite[Lemma~9]{BuffaH_03_GBE}).
This proves {\bf (A)}.

Before we can define the discrete counterpart of decomposition (\ref{decomp}),
we need to find a suitable projector onto the space of Raviart-Thomas
surface elements. Stability of the discrete decomposition will follow
from stability properties of the Raviart-Thomas interpolation on anisotropic elements,
which is the subject of the next section.

\newcommand{\hPi}{\hat\Pi}

\section{\rblf{Raviart-Thomas interpolation on anisotropic\\ elements}} \label{sec_rt_stab}
\setcounter{equation}{0}

\rblf{In this section, we establish new stability properties of the Raviart-Thomas interpolant
on anisotropic elements.
We will also prove the corresponding interpolation error estimates.}

\rev{In the context of the finite element method, the analysis of interpolation operators
on anisotropic elements can be found in~\cite{Apel_99_AFE, AcostaADL_11_EER}.
For the Raviart-Thomas interpolation (see, e.g.,~\cite[Section~3]{AcostaADL_11_EER}), the main idea is
to study componentwise stability of the interpolant on a reference element $\hat K$.
For sufficiently regular vector fields, this study relies on the fact
that the standard (scalar) trace operator is well defined for functions in $W^{1,p}(\hat K)$ for any $p>1$.
In our BEM application, however, the stability result is needed for low-regular vector
fields living in fractional Sobolev spaces $\bH^{s}(\hat K) \cap \bH(\div,\hat K)$ with $0 < s \le 1/2$.
For such vector fields, the trace of the normal component only exists in a weak sense.
Therefore, instead of the standard trace argument used in~\cite{AcostaADL_11_EER},
we use Green's formula~(\ref{serge9/05:2}) and consistently employ
the anisotropic seminorms defined below.
More precisely, for $s \in (0,1/2]$ we introduce the $H^{s}$-seminorms of anisotropic type.}
On the reference square $\hat Q=(0,1)^2$ these are defined as follows:
\begin{eqnarray*}
      |u|_{AH_1^{s}(\hat Q)}^2 =
      \int_0^1 |u(\cdot, x_2)|_{H^{s}(0,1)}^2\, dx_2,
      \\
      |u|_{AH_2^{s}(\hat Q)}^2 =
      \int_0^1 |u(x_1,\cdot)|_{H^{s}(0,1)}^2\, dx_1.
\end{eqnarray*}
\rblf{These definitions are meaningful} for all $u\in H^{s}(\hat Q)$
due to~\rblf{\cite[Theorem~10.2]{LionsMagenes}} \rblf{which} yields that
\be \label{serge9/05:4}
    \|u\|_{H^{s}(\hat Q)} \simeq \|u\|_{0,\hat Q} + |u|_{AH_1^{s}(\hat Q)} +
    |u|_{AH_2^{s}(\hat Q)}.
\ee
On the reference triangle
\rblf{$\hat T = \{(x_1, x_2);\; 0 < x_1 < 1,\ 0 < x_2 < x_1\}$,
the following seminorms}
\beas
     |u|_{AH_1^{s}(\hat T)}^2=
     \int_0^1 |u(\cdot, x_2)|_{H^{s}(x_2,1)}^2\, dx_2,\\
     |u|_{AH_2^{s}(\hat T)}^2=
     \int_0^1 |u(x_1,\cdot)|_{H^{s}(0,x_1)}^2\, dx_1
\eeas
\rblf{are also well defined} for all $u\in H^{s}(\hat T)$.
\rblf{Indeed,} by Theorem 1.4.3.1 of \cite{Grisvard_85_EPN}
there exists a continuous linear operator (called the extension operator)
$\sfE: H^{s}(\hat T) \to H^{s}(\hat Q)$ such that
\[
  \sfE u|_{\hat T}=u,\quad \forall u\in H^{s}(\hat T).
\]
Hence, applying (\ref{serge9/05:4}) to $\sfE u$, we have
\be \label{serge9/05:11}
    |u|_{AH_1^{s}(\hat T)}+
    |u|_{AH_2^{s}(\hat T)}\, \lesssim\, \|u\|_{H^{s}(\hat T)}.
\ee

We recall that the Raviart-Thomas interpolant
\rblf{$\Pi_{\rm RT}\bu$ is well-defined for any} $\bu \in \bH^{s}(K)$ ($s>0$)
such that $\div\,\bu \in L^{2}(K)$, where $K$ is any triangle or rectangle.
Indeed, for such {\rblf vector fields}
the following Green's formula has a meaning (\rblf{see, e.g.,~\cite[Lemma~2.1]{BespalovH_11_NHC}})
\be \label{serge9/05:2}
    (\bu, \nabla \varphi)_{0,K} + \int_K \div\,\bu\,\varphi =
    (\bu \cdot \bn, \varphi)_{0,\partial K},\quad
    \forall \varphi\in H^{1-\eps}(K),
\ee
with $\eps\in (0,s)$ \rblf{and $\bn$ denoting the outward unit normal to $\partial K$}.
Hence, taking $\varphi \in H^{1-\eps}(K)$ such that
$\varphi=1$ on \rblf{the edge $e \subset \partial K$ and $\varphi=0$ on $\partial K{\setminus}e$},
\rblf{we can} define
$(\bu\cdot \bn, 1)_{0,e} := (\bu \cdot \bn, \varphi)_{0,\partial K}$.

In what follows, we will denote by $\hPi_{\rm RT}$ the Raviart-Thomas interpolation
operator on the reference element $\hat K$ ($\hat K = \hat Q$ or $\hat T$).

\subsection{\rblf{The reference square}}

\rblf{On the reference square $\hat Q$,
we denote by $\hat e_1$ and $\hat e_3$ the edges parallel to the $x_1$-axis,
and by $\hat  e_2$ and $\hat  e_4$ the edges parallel to the $x_2$-axis (see Figure~\ref{fig_2}).}
We recall \rblf{from}~\cite{BrezziF_91_MHF} that the lowest order Raviart-Thomas
elements on $\hat Q$ are defined as
\[
  \RT_0(\hat Q) = \left\{(a + c x_1, b + d x_2)^\top;\; a,b,c,d\in {\field{R}}\right\},
\]
and that the associated degrees of freedom are given by
\[
  \int_{\hat e_i} \bu \cdot \bn\,ds,\quad i=1,2,3,4.
\]

\begin{figure}[!htb]
\begin{center}
\includegraphics[width=0.75\textwidth]{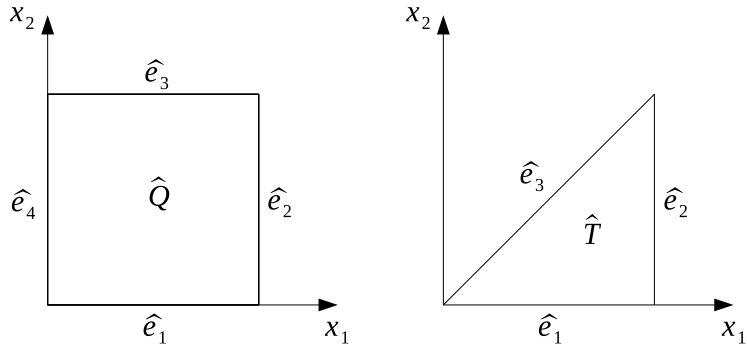}
\addtolength{\abovecaptionskip}{-2em}
\end{center}
\caption{The reference square $\hat Q$ and the reference triangle $\hat T$.}
\label{fig_2}
\addtolength{\abovecaptionskip}{1em}
\end{figure}

\begin{theorem} \label{thm_RT_square}
\rblf{{\rm (i)} For all $\bu \in \bH^{s}(\hat Q)$ with $s > 1/2$,
we have for $l=1,\,2$}
\be \label{serge9/05:1(i)}
    \rblf{\| (\hPi_{\rm RT} \bu)_l\|_{0,\hat Q}\lesssim
	  \|u_l\|_{H^{s}(\hat Q)}.}
\ee
\rblf{{\rm (ii)} For all  $\bu \in \bH^{s}(\hat Q) \cap \bH(\div,\hat Q)$ with $0 < s \le 1/2$,
we have for $l=1,\,2$}
\be \label{serge9/05:1}
    \| (\hPi_{\rm RT} \bu)_l\|_{0,\hat Q}\lesssim
	  \|u_l\|_{H^{s}(\hat Q)} + |u_{l+1}|_{AH_{l+1}^{s}(\hat Q)}
	  + \rblf{\|\div\,\bu\|_{0,\hat Q}}\ \ \hbox{\rm (mod. $2$).}
\ee
\end{theorem}
\begin{proof}
By symmetry, it \rblf{will suffice} to prove \rblf{both statements} for $l=1$.
\rblf{One has}
\[
  \hPi_{\rm RT} \bu = (a + c x_1, b + d x_2)^\top\quad
  \hbox{with $a,b,c, d\in {\field{R}}$}.
\]
Hence
\[
  \rblf{\| (\hPi_{\rm RT} \bu)_1 \|_{0,\hat Q} = \| a + c x_1 \|_{0,\hat Q} \lesssim |a|+|c|},
\]
and \rblf{it remains} to estimate $|a|$ and $|c|$. \rblf{Observe} that
\[
     a = -(\bu \cdot \bn, 1)_{0,\hat e_4}\quad \hbox{and}\quad
     a + c = (\bu \cdot \bn, 1)_{0,\hat e_2}.
\]
Again by symmetry, it \rblf{will suffice} to estimate $(\bu \cdot \bn, 1)_{0,\hat e_4}$.

\rblf{If $\bu \in \bH^{s}(\hat Q)$ with $s > 1/2$, then the trace of $u_1$ on $\hat e_4$
is well-defined, and}
\[
  \rblf{|a| = |(\bu \cdot \bn, 1)_{0,\hat e_4}| = |(u_1, 1)_{0,\hat e_4}| \lesssim
  \|u_1\|_{H^{s-1/2}(\hat e_4)} \lesssim \|u_1\|_{H^s(\hat Q)}.}
\]
\rblf{This proves statement (i) (for $l=1$).}

\rblf{Now, let us consider $\bu \in \bH^{s}(\hat Q) \cap \bH(\div,\hat Q)$ for $0 < s \le 1/2$.
In order to estimate $(\bu \cdot \bn, 1)_{0,\hat e_4}$ in this case,}
we fix \rblf{a function} $\varphi\in H^{1-\eps}(\hat Q)$,
$\eps\in (0,s)$, such that $\varphi=1$ on $\hat e_4$ and
$\varphi=0$ on \rblf{$\partial \hat Q {\setminus} \hat e_4$}. 
\rblf{Then by Green's formula (\ref{serge9/05:2}) we have}
\be
    (\bu \cdot \bn, 1)_{0,\hat e_4} =
    (\bu, \nabla \varphi)_{0,\hat Q}\,
    \rblf{+ \int_{\hat Q} \div\bu\, \varphi} =
    \rblf{(u_1,\partial_1 \varphi)_{0,\hat Q} +
    (u_2,\partial_2 \varphi)_{0,\hat Q} + \int_{\hat Q} \div\bu\, \varphi}.
    \label{serge9/05:8}
\ee
For the term $(u_1,\partial_1 \varphi)_{0,\hat Q}$, we first use a standard duality argument,
\[
  |(u_1,\partial_1 \varphi)_{0,\hat Q}| \leq
  \|u_1\|_{H^{\eps}(\hat Q)}\, \|\partial_1 \varphi\|_{H^{-\eps}(\hat Q)},
\]
and by the continuity property of $\partial_1: H^{1-\eps}(\hat Q) \to H^{-\eps}(\hat Q)$
(see Theorem 1.4.4.6 of \cite{Grisvard_85_EPN}) we find
\be \label{serge9/05:3}
    |(u_1,\partial_1 \varphi)_{0,\hat Q}| \lesssim
    \|u_1\|_{H^{s}(\hat Q)}\, \|\varphi\|_{H^{1-\eps}(\hat Q)}.
\ee
The term $(u_2,\partial_2 \varphi)_{0,\hat Q}$ requires more \rblf{subtle} analysis.
First by (\ref{serge9/05:4}) and Fubini's theorem, we can write (hereafter, $\hat I = (0,1)$)
\be \label{serge9/05:5}
    (u_2,\partial_2 \varphi)_{0,\hat Q} =
    \int_0^1 (u_2(x_1,\cdot),\partial_2 \varphi(x_1,\cdot))_{0,\hat I} \,dx_1.
\ee
Now we use a density argument to show that  
\be \label{serge9/05:6}
    \rkserge{\int_0^1  h(x_1)(1,\partial_2 \varphi(x_1,\cdot))_{0,\hat I}\,dx_1 = 0,}
\ee
\rkserge{for all $h\in L^2(\hat I)$}.
Indeed, fix a sequence of smooth function $\varphi_n$ such that
\[
  \varphi_n\to \varphi \hbox{ in } H^{1-\eps}(\hat Q) \hbox{ as } n\to\infty.
\]
Then for all $x_1 \in \hat I = (0,1)$, we have
\[
  (1,\partial_2 \varphi_n(x_1,\cdot))_{0,\hat I} =
  \int_0^1 \partial_2 \varphi_n(x_1,x_2)\,dx_2=\varphi_n(x_1,1)-\varphi_n(x_1,0).
\]
\rkserge{For any $h\in L^2(\hat I)$, multiplying   this identity by $h$ and 
integrating the result in $x_1 \in (0,1)$}, we obtain
\[
 \rkserge{ \int_0^1 h(x_1)(1,\partial_2 \varphi_n(x_1,\cdot))_{0,\hat I}\,dx_1 =
  \int_{\hat e_3} h\varphi_n-\int_{\hat  e_1} h \varphi_n.}
\]
\rblf{Hence, as $n \to \infty$}
we find that
\[
 \rkserge{ \int_0^1 h(x_1)(1,\partial_2 \varphi(x_1,\cdot))_{0,\hat I}\,dx_1 =
  \int_{\hat  e_3} h \varphi-\int_{\hat  e_1} h\varphi,}
\]
which proves (\ref{serge9/05:6}) \rblf{by recalling} that $\varphi=0$ on $\hat e_1$ and $\hat e_3$.

Coming back to (\ref{serge9/05:5}) and using (\ref{serge9/05:6}), we have
\[
  (u_2,\partial_2 \varphi)_{0,\hat Q} =
  \int_0^1 \big( u_2(x_1,\cdot)-{\CM}_{\hat I}(u_2(x_1,\cdot)),
                 \partial_2 \varphi(x_1,\cdot) \big)_{0,\hat I} \,dx_1,
\]
where ${\CM}_{\hat I}(u_2(x_1,\cdot)) = \int_0^1u_2(x_1,x_2)\,dx_2$
is the mean of $u_2(x_1,\cdot)$ on $\hat I = (0,1)$
\rkserge{(clearly ${\CM}_{\hat I}(u_2(x_1,\cdot)) \in L^2(\hat I)$)}.
At this stage we first use a duality argument and then again Theorem 1.4.4.6 of \cite{Grisvard_85_EPN}
to obtain
\bea
     |(u_2,\partial_2 \varphi)_{0,\hat Q}|
     & \leq &
     \int_0^1\| u_2(x_1,\cdot)-{\CM}_{\hat I}(u_2(x_1,\cdot)\|_{H^{\eps}(0,1)}
     \|\partial_2 \varphi(x_1,\cdot)\|_{H^{-\eps}(0,1)} \,dx_1
     \nonumber
     \\
     &\lesssim &
     \int_0^1\| u_2(x_1,\cdot)-{\CM}_{\hat I}(u_2(x_1,\cdot)\|_{H^{\eps}(0,1)}
     \|\varphi(x_1,\cdot)\|_{H^{1-\eps}(0,1)} \,dx_1.
     \label{alex30/07:1}
\eea
\rblf{We use Friedrichs' inequality to estimate}
\[
  \| u_2(x_1,\cdot)-{\CM}_{\hat I}(u_2(x_1,\cdot)\|_{H^{\eps}(0,1)} \leq
  \| u_2(x_1,\cdot)-{\CM}_{\hat I}(u_2(x_1,\cdot)\|_{H^{s}(0,1)} \lesssim
  | u_2(x_1,\cdot)|_{H^{s}(0,1)}.
\]
Using this estimate in (\ref{alex30/07:1}) and applying the Cauchy-Schwarz inequality we arrive at
\be \label{serge9/05:7}
    |(u_2,\partial_2 \varphi)_{0,\hat Q}| \lesssim 
    |u_{2}|_{AH_{2}^{s}(\hat Q)}
    \Big(\int_0^1
    \| \varphi(x_1,\cdot)\|_{H^{1-\eps}(0,1)}^2 \,dx_1\Big)^\frac12.
\ee
\rblf{The last term on the right-hand side of (\ref{serge9/05:8})
is estimated by applying the Cauchy-Schwarz inequality:
\[
  \int_{\hat Q} \div\,\bu \varphi \le \|\div\,\bu\|_{0,\hat Q}\, \|\varphi\|_{0,\hat Q}.
\]
}
Using this estimate and \rblf{inequalities (\ref{serge9/05:3}), (\ref{serge9/05:7}) in (\ref{serge9/05:8})},
we obtain (\ref{serge9/05:1}) (for $l=1$).
\end{proof}

\begin{corollary} \label{coroserge16/05:1}
\rblf{{\rm (i)} For all $\bu \in \bH^{s}(\hat Q)$ with $s > 1/2$, we have for $l=1,\,2$}
\[
    \rblf{\|u_l -(\hPi_{\rm RT} \bu)_l\|_{0,\hat Q}\lesssim
    |u_l|_{H^{s}(\hat Q)}.}
\]
\rblf{{\rm (ii)} For all $\bu \in \bH^{s}(\hat Q) \cap \bH(\div,\hat Q)$ with $0 < s \le 1/2$,
we have for $l=1,\,2$}
\[
    \|u_l -(\hPi_{\rm RT} \bu)_l\|_{0,\hat Q}\lesssim
    |u_l|_{H^{s}(\hat Q)} + |u_{l+1}|_{AH_{l+1}^{s}(\hat Q)} \rblf{+ \|\div\,\bu\|_{0,\hat Q}}\ \ 
		\hbox{\rm (mod. 2).}
\]
\end{corollary}
\begin{proof}
It is sufficient to prove only statement (ii).
For $l=1$, we take $\widetilde\bu = \bu - (\CM_{\hat Q} u_1, 0)^\top$,
\rblf{where $\CM_{\hat Q} u_1 = \int_{\hat Q} u_1$.
One has}
\[
  \widetilde\bu - \hPi_{\rm RT} \widetilde\bu = \bu - \hPi_{\rm RT} \bu,\qquad
  \div\,\widetilde\bu = \div\,\bu,
\]
and
\[
  |u_1|_{H^{1/2}(\hat Q)} = |\tilde u_1|_{H^{1/2}(\hat Q)} \simeq \|\tilde u_1\|_{H^{1/2}(\hat Q)}.
\]
The assertion then follows by applying estimate (\ref{serge9/05:1}) to $\widetilde\bu$.
The proof is analogous for $l=2$.
\end{proof}

\begin{corollary} \label{coroalex25/06:1}
\rblf{Let $s \in (0,1/2]$.
For all  $\bu \in \bH^{s}(\hat Q)$ such that $\div\,\bu \in {\field{R}}$,
we have for $l=1,\,2$}
\be \label{alex25/06:2}
    \rblf{\| (\hPi_{\rm RT} \bu)_l\|_{0,\hat Q}\lesssim
    \|u_l\|_{H^{s}(\hat Q)} + |u_{l+1}|_{AH_{l+1}^{s}(\hat Q)}\ \ \hbox{\rm (mod. $2$)}}
\ee
\rblf{and}
\be \label{alex25/06:3}
    \rblf{\|u_l -(\hPi_{\rm RT}\bu)_l\|_{0,\hat Q}\lesssim
    |u_l|_{H^{s}(\hat Q)} + |u_{l+1}|_{AH_{l+1}^{s}(\hat Q)}\ \ 
    \hbox{\rm (mod. $2$).}}
\ee
\end{corollary}
\begin{proof}
\rblf{
Note that the function $\varphi \in H^{1-\eps}(\hat Q)$ in the proof of Theorem~\ref{thm_RT_square}~(ii)
can be chosen to have zero average on $\hat Q$, i.e., $\int_{\hat Q} \varphi = 0$.
Indeed, if this is not the case then we fix $\psi \in C^{\infty}_0(\hat Q)$ such that 
$\int_{\hat Q}\psi=1$, and consider
$\varphi-\big(\int_{\hat Q}\varphi\big) \psi \in H^{1-\eps}(\hat Q)$ that has the same values
on $\partial\hat Q$ as $\varphi$ and also has zero average on $\hat Q$.
Then, for $\div\,\bu \in {\field{R}}$, the last term on the right-hand side of (\ref{serge9/05:8}) vanishes,
and (\ref{alex25/06:2}) (resp., (\ref{alex25/06:3}))
follows by the same arguments as in the proof of Theorem~\ref{thm_RT_square}~(ii)
(resp., Corollary~\ref{coroserge16/05:1}~(ii)).
}
\end{proof}

\subsection{\rblf{The reference triangle}}

\rblf{On the references triangle $\hat T$,
we denote by $\hat e_1$ (resp. $\hat  e_2$) the edge on the $x_1$-axis
(resp., parallel to the $x_2$-axis),
and by $\hat e_3$ the oblique edge (see Figure~\ref{fig_2}).}
We recall \rblf{from} \cite{BrezziF_91_MHF} that the lowest order Raviart-Thomas elements
on $\hat T$ are defined as
\[
  \RT_0(\hat T) = \left\{(a, b)^\top + c(x_1, x_2)^\top;\; a,b,c \in {\field{R}}\right\},
\]
and that the associated degrees of freedom are given by
\[
  \int_{\hat e_i} \bu \cdot \bn\,ds,\quad i = 1,2,3.
\]

\begin{theorem} \label{thm_RT_triangle}
\rblf{{\rm (i)} For all $\bu \in \bH^{s}(\hat T) \cap \bH(\div,\hat T)$ with $s > 1/2$,
we have for $l=1,\,2$}
\be \label{serge9/05:10(i)}
    \rblf{\|(\hPi_{\rm RT} \bu)_l\|_{0,\hat T} \lesssim
    \|u_l\|_{H^{s}(\hat T)} + \|\div\,\bu\|_{0,\hat T}.}
\ee
\rblf{{\rm (ii)} For all $\bu \in \bH^{s}(\hat T) \cap \bH(\div,\hat T)$ with $0 < s \le 1/2$,
we have for $l=1,\,2$}
\be \label{serge9/05:10}
    \|(\hPi_{\rm RT} \bu)_l\|_{0,\hat T} \lesssim
    \|u_l\|_{H^{s}(\hat T)} + |u_{l+1}|_{AH_{l+1}^{s}(\hat T)} +
    \rblf{\|\div\,\bu\|_{0,\hat T}}\ \ \hbox{\rm (mod. $2$).}
\ee
\end{theorem}
\begin{proof}
\rblf{We will prove both statements for $l=1$. The proof is analogous in the case $l=2$.
One has 
$\hPi_{\rm RT} \bu = (a, b)^\top + c(x_1, x_2)^\top$
with $a,b,c \in {\field{R}}$. Hence}
\[
  \rblf{\| (\hPi_{\rm RT} \bu)_1 \|_{0,\hat T} = \| a + c x_1 \|_{0,\hat T} \lesssim |a|+|c| \le |a+c| + 2|c|}.
\]
Observe that
\[
  2c = \div\,\hPi_{\rm RT}\bu = \rblf{2\int_{\hat T}\div\,\bu}.
\]
\rblf{By the Cauchy-Schwarz inequality this yields}
\[
  |c| \lesssim \|\div\,\bu\|_{0,\hat T}.
\]
Therefore in order to estimate $\| (\hPi_{\rm RT} \bu)_1\|_{0,\hat T}$
\rblf{one needs to bound $|a+c| = \left|(\bu \cdot \bn, 1)_{0,\hat e_2}\right|$}.

\rblf{If $\bu \in \bH^{s}(\hat T)$ with $s > 1/2$, then we can use trace theorem
to estimate}
\[
  \rblf{|a + c| = |(\bu \cdot \bn, 1)_{0,\hat e_2}| =
  |(u_1, 1)_{0,\hat e_2}| \lesssim \|u_1\|_{H^s(\hat T)}.}
\]
\rblf{This proves statement (i) (for $l=1$).}

\rblf{Now, let us consider $\bu \in \bH^{s}(\hat T) \cap \bH(\div,\hat T)$ for $0 < s \le 1/2$.
In order to estimate $(\bu \cdot \bn, 1)_{0,\hat e_2}$ in this case,}
we fix a function $\varphi\in H^{1-\eps}(\hat T)$, $\eps\in (0,s)$,
such that $\varphi=1$ on $\hat e_2$ and $\varphi=0$ on \rblf{$\partial T {\setminus} \hat e_2$}. 
Then by Green's formula (\ref{serge9/05:2}) we have
\bea
    (\bu \cdot \bn, 1)_{0,\hat e_2} =
    (u, \nabla \varphi)_{0,\hat T}\, \rblf{+ \int_{\hat T} \div\bu\, \varphi} =
    \rblf{(u_1,\partial_1 \varphi)_{0,\hat T} +
          (u_2,\partial_2 \varphi)_{0,\hat T} + \int_{\hat T} \div\bu\, \varphi.
         }
    \label{serge9/05:12}
\eea
As before \rblf{(cf. (\ref{serge9/05:3})), it is easy to} show that
\be \label{serge9/05:13}
     |(u_1,\partial_1 \varphi)_{0,\hat T}| \lesssim
     \|u_1\|_{H^{s}(\hat T)}\, \|\varphi\|_{H^{1-\eps}(\hat T)}.
\ee
The term  $(u_2,\partial_2 \varphi)_{0,\hat T}$
is treated \rblf{as in} the case of the reference square
except that \rblf{the interval $\hat I = (0,1)$} in the $x_2$-variable will be replaced
by \rblf{the interval $I(x_1) := (0,x_1)$}, but the arguments remain mainly unchanged
because $x_1 < 1$.
More precisely, first by (\ref{serge9/05:11}) and Fubini's theorem, we may write
\be \label{serge9/05:15}
    (u_2,\partial_2 \varphi)_{0,\hat T} =
    \int_0^1
    (u_2(x_1,\cdot),\partial_2 \varphi(x_1,\cdot))_{0,I(x_1)}
    \,dx_1.
\ee
With a property \rblf{similar to} (\ref{serge9/05:6}) we deduce that
\[
  (u_2,\partial_2 \varphi)_{0,\hat T} =
  \int_0^1
  \big(u_2(x_1,\cdot)-{\CM}_{I(x_1)}(u_2(x_1,\cdot)),
       \partial_2 \varphi(x_1,\cdot)\big)_{0,I(x_1)}
  \,dx_1,
\]
where ${\CM}_{I(x_1)}(u_2(x_1,\cdot))=\frac{1}{x_1}\int_0^{x_1}u_2(x_1,x_2)\,dx_2$
is the mean of $u_2(x_1,\cdot)$ on $I(x_1) = (0,x_1)$. 
Then a duality argument yields
\be \label{serge9/05:16}
    |(u_2,\partial_2 \varphi)_{0,\hat T}| \leq
    \int_0^1\| u_2(x_1,\cdot)-{\CM}_{I(x_1)}(u_2(x_1,\cdot))\|_{H^{\eps}(I(x_1))}
    \|\partial_2 \varphi(x_1,\cdot)\|_{H^{-\eps}(I(x_1))} \,dx_1.
\ee
\rblf{Using a scaling argument and Friedrichs' inequality we estimate}  
\bea \label{serge12/05:1}
     \nonumber
     \| u_2(x_1,\cdot)-{\CM}_{I(x_1)}(u_2(x_1,\cdot))\|_{H^{\eps}(I(x_1))}
     & \leq &
     \| u_2(x_1,\cdot)-{\CM}_{I(x_1)}(u_2(x_1,\cdot))\|_{H^{s}(I(x_1))}\qquad
     \\
     & \leq &
     C | u_2(x_1,\cdot)|_{H^{s}(I(x_1))},
\eea
with $C>0$ independent of $x_1\in (0,1)$.

For the second factor in the integrand in (\ref{serge9/05:16}),
in order to apply Theorem 1.4.4.6 of \cite{Grisvard_85_EPN}
on a fixed domain, we first notice that
$\varphi(x_1,\cdot)\in H^{1-\eps}_0(I(x_1))$ \rblf{a. e. in $(0,1) \ni x_1$}.
Therefore, \rblf{for almost all $x_1 \in (0,1)$ there exists a sequence of functions}
$\varphi_n \in C^{\infty}_0(I(x_1))$, \rblf{$n=1,2,\ldots$, such that
$\varphi_n \to \varphi(x_1,\cdot)$ in $H^{1-\eps}_0(I(x_1))$ as $n \to \infty$ and}
\[
  \int_0^{x_1} \partial \varphi_n(x_2)\,dx_2 = 0\quad \forall n = 1,2,\ldots.
\]
Using a scaling argument, we deduce that
\[
     \|\partial \varphi_n\|_{H^{-\eps}(I(x_1))} =
     \stack{v\ne 0}{\sup_{v\in H^{\eps}(I(x_1))}}
     \frac{\int_0^{x_1} \partial \varphi_n\, v\, dx_2}{\|v\|_{H^{\eps}(I(x_1))}} \leq
     x_1^{-\frac12+\eps} \stack{\hat v\ne 0}{\sup_{\hat v\in H^{\eps}(0,{1})}}
     \frac{\int_0^{1} \partial \hat \varphi_n\, \hat v\, d\hat x_2}{|\hat v|_{H^{\eps}(0,1)}}.
\]
As $\partial \hat \varphi_n$ \rblf{has zero average}, we can estimate
\beas
     \|\partial \varphi_n\|_{H^{-\eps}(I(x_1))}
     & \leq &
     x_1^{-\frac12+\eps}
     \stack{\hat v\ne 0,\;\int_0^1\hat v=0}{\sup_{\hat v\in H^{\eps}(0,{1})}}
     \frac{\int_0^{1} \partial \hat \varphi_n \hat v\, d\hat x_2}{|\hat v|_{H^{\eps}(0,1)}}
     \\
     & \lesssim &
     x_1^{-\frac12+\eps}
     \stack{\hat v\ne 0,\;\int_0^1\hat v=0}{\sup_{\hat v\in H^{\eps}(0,{1})}}
     \frac{\int_0^{1} \partial \hat \varphi_n \hat v\, d\hat x_2}{\|\hat v\|_{H^{\eps}(0,1)}}\;
     \rblf{ \le x_1^{-\frac12+\eps} \|\partial \hat\varphi_n\|_{H^{-\eps}(0,1)}}.
\eeas
Hence, by Theorem 1.4.4.6 of \cite{Grisvard_85_EPN} we prove that
\beas
     \|\partial \varphi_n\|_{H^{-\eps}(I(x_1))} \lesssim
     x_1^{-\frac12+\eps}
     \rblf{\|\hat \varphi_n\|_{H^{1-\eps}(0,1)} \lesssim
           x_1^{-\frac12+\eps} |\hat \varphi_n|_{H^{1-\eps}(0,1)}}.
\eeas
\rblf{Mapping back to the interval $I(x_1) = (0,x_1)$ we have}
\beas
     \|\partial \varphi_n\|_{H^{-\eps}(I(x_1))} \lesssim
     | \varphi_n|_{H^{1-\eps}(I(x_1))}.
\eeas
\rblf{As $n \to \infty$ we find}
\be \label{alex28/07:1}
     \|\partial_2 \varphi(x_1,\cdot)\|_{H^{-\eps}(I(x_1))} \leq
     C_1 |\varphi(x_1,\cdot)|_{H^{1-\eps}(I(x_1))}\quad
     \rblf{\hbox{a. e. on $(0,1) \ni x_1$}},
\ee
with $C_1>0$ independent of $x_1$.

\rblf{Using estimates (\ref{serge12/05:1}) and (\ref{alex28/07:1}) in (\ref{serge9/05:16})
we arrive at}
\[
  \big|(u_2,\partial_2 \varphi)_{0,\hat T}\big| \lesssim
  \int_0^1| u_2(x_1,\cdot)|_{H^{s}(0,x_1)}
  \|\varphi(x_1,\cdot)\|_{H^{1-\eps}(0,x_1)} \,dx_1.
\]
Then the Cauchy-Schwarz inequality yields
\be\label{serge9/05:17}
   \big|(u_2,\partial_2 \varphi)_{0,\hat T}\big| \lesssim 
   |u_{2}|_{AH_{2}^{s}(\hat T)}
   \Big(\int_0^1
        \|\varphi(x_1,\cdot)\|^2_{H^{1-\eps}(0,x_1)} \,dx_1
   \Big)^\frac12.
\ee
\rblf{The last term on the right-hand side of (\ref{serge9/05:12})
is estimated by using the Cauchy-Schwarz inequality:}
\[
  \int_{\hat T} \div\bu\, \varphi \le \|\div\,\bu\|_{0,\hat T}\, \|\varphi\|_{0,\hat T}.
\]
Using this estimate and \rblf{inequalities (\ref{serge9/05:13}), (\ref{serge9/05:17}) in (\ref{serge9/05:12})},
we obtain (\ref{serge9/05:10}) (for $l{=}1$).
\end{proof}

\begin{corollary} \label{coroserge16/05:2}
\rblf{{\rm (i)} For all  $\bu \in \bH^{s}(\hat T) \cap \bH(\div,\hat T)$ with $s > 1/2$,
we have for $l=1,\,2$}
\[
    \rblf{\|u_l -(\hPi_{\rm RT} \bu)_l\|_{0,\hat T} \lesssim
    |u_l|_{H^{s}(\hat T)} + \|\div\,\bu\|_{0,\hat T}.}
\]
\rblf{{\rm (ii)} For all $\bu \in \bH^{s}(\hat T) \cap \bH(\div,\hat T)$ with $0 < s \le 1/2$,
we have for $l=1,\,2$}
\[
    \|u_l -(\hPi_{\rm RT} \bu)_l\|_{0,\hat T} \lesssim
    |u_l|_{H^{s}(\hat T)}+|u_{l+1}|_{AH_{l+1}^{s}(\hat T)} + \|\div\,\bu\|_{0,\hat T}\quad
    \hbox{\rm (mod. 2).}
\]
\end{corollary}

\rblf{The proof of this statement is similar to the proof of Corollary \ref{coroserge16/05:1}.}

\begin{counterex} \label{cntrex_1}
\rblf{Here we provide a counterexample which demonstrates that
for low-regular vector fields
the terms $|u_{l+1}|_{AH_{l+1}^{s}(\hat Q)}$ in~{\rm (\ref{serge9/05:1})}
and $|u_{l+1}|_{AH_{l+1}^{s}(\hat T)}$ in~{\rm (\ref{serge9/05:10})} cannot be omitted.
This is in contrast to the case of sufficiently-regular fields
in~{\rm (\ref{serge9/05:1(i)})} and~{\rm (\ref{serge9/05:10(i)})}
(see also} Lemma~{\rm 3.3} in {\rm \cite{AcostaADL_11_EER}}).
\rblf{In particular}, if we assume that
\be \label{serge9/05:10wrong}
    \|(\hPi_{\rm RT} \bu)_2\|_{0,\hat T} \lesssim
    \|u_2\|_{H^{1/2}(\hat T)} + \|\div\,\bu\|_{0,\hat T}
\ee
for all \rblf{$\bu \in \bH^{\frac12}(\hat T) \cap \bH(\div,\hat T)$},
then we will arrive at a contradiction.
Indeed, inspired by Example {\rm 2.6} of {\rm \cite{Apel_99_AFE}},
we \rblf{define} on $\hat T \times (0,1)$
\[
  v^\eps(x_1,x_2, x_3)= (x_1-1) w^\eps(x_2, x_3)\ \
  \hbox{with }\ \ 
  w^\eps(x_2, x_3)=\min\Big\{1, \eps\log\log\frac{e}{r}\Big\}\ \
  \hbox{for any $\eps>0$}.
\]
Here, $r=(x_2^2+x_3^2)^\frac12$, and $e$ is the Euler number.
Taking $\nabla v^\eps \times \bnu$ on $\hat T$ \rblf{(here $\bnu = (0,0,-1)$)}, we find
a divergence-free vector field
\[
  \bu^\eps(x_1,x_2) = \big((1-x_1) \partial_2 w^\eps(x_2, 0), w^\eps(x_2, 0)\big)^\top.
\]
Simple calculations show that 
$\hPi_{\rm RT} \bu^\eps = (0,1)^\top$, and by trace theorem we have
\beas
     \|(\bu^{\eps})_2\|_{H^{1/2}(\hat T)} = \|w^{\eps}(x_2, 0)\|_{H^{1/2}(\hat T)} \lesssim
     \|w^\eps \|_{H^1(\hat T \times (0,1))}.
\eeas
Since \rblf{$\|w^\eps \|_{H^1(\hat T \times (0,1))} \to 0$ as $\eps \to 0$
(see~{\rm \cite[Example~2.6]{Apel_99_AFE}})},
we \rblf{conclude} that \rblf{$(\bu^{\eps})_2 \to 0$} in $H^\frac12(\hat T)$ \rblf{as $\eps \to 0$}.
This seems to contradict {\rm (\ref{serge9/05:10wrong})}
but not directly because the first component of $\bu^\eps$ is not in $H^\frac12(\hat T)$.
Hence, in order to arrive at a contradiction, we need to
show that if {\rm (\ref{serge9/05:10wrong})} holds for all
\rblf{$\bu \in \bH^{\frac12}(\hat T) \cap \bH(\div,\hat T)$},
then $\bu^\eps$ satisfies {\rm (\ref{serge9/05:10wrong})} (with a constant indepedent of $\eps$).
Indeed, for a fixed $\eps>0$, as \rblf{$w^\eps(\cdot,0) \in H^1(0,1)$},
we can consider a sequence of smooth functions $w_n \in C^{\infty}([0,1])$ such that
\[
  w_n \to w^\eps(\cdot,0)\hbox{ in } H^1(0,1) \hbox{ as } n\to \infty.
\]
Then we define 
\[
  \bu_n(x_1,x_2) = ((1-x_1) \rblf{\partial w_n}(x_2), w_n(x_2))^\top.
\]
\rblf{One has $\bu_n \in \bH^{\frac12}(\hat T)$ and $\div\,\bu_n = 0$}.
\rblf{Moreover,}
$\hPi_{\rm RT} \bu_n \to \hPi_{\rm RT} \bu^{\rblf{\eps}}$ as $n \to \infty$.
\rblf{Therefore,} applying estimate {\rm (\ref{serge9/05:10wrong})} to $\bu_n$
and \rblf{letting $n \to \infty$,}
we conclude that $\bu^{\eps}$ satisfies {\rm (\ref{serge9/05:10wrong})}.
\end{counterex}

\begin{remark}
\rblf{By Counterexample~{\rm \ref{cntrex_1}}} we can easily show that 
a result similar to Lem\-ma~{\rm 3.3} in {\rm \cite{AcostaADL_11_EER}}
for the Nédélec interpolant
\rblf{on the tetrahedron
      $\hat T_3 = \{(x_1,x_2, x_3) \in {\field{R}}^3;\;
                    x_i > 0,\ i=1,2,3 \hbox{ and } 0 < x_1+x_2+x_3 < 1\}$
}
is not valid.
In other words, the anisotropic estimate
\[
    \| (\hPi_{\rm Ned} \bv)_l\|_{0,\hat T_3} \lesssim \|v_l\|_{\bH^1(\hat T_3)} + \|\bcurl\,\bv\|_{0,\hat T_3},\quad
    l=1,2,3
\]
\rblf{does not hold for all  $\bv \in \bH^1(\hat T_3) = \big(H^1(\hat T_3)\big)^3$}.
\end{remark}

\subsection{\rblf{Anisotropic elements}} \label{subsec_RT_real_element}

In this subsection we will denote the functions
on the elements $K$ and $\hat K$ by $\bu$ and $\hat\bu$, respectively.
Analogous notation will be used for coordinates (e.g., $\bx \in K$ and $\hat\bx \in \hat K$)
and for differential operators (e.g., $\div$ and $\hat\div$).

First, let us prove the following auxiliary result.

\begin{lemma} \label{lm_H^{-1/2}-scaling}
Let $K$ be the image of the reference element $\hat K$ ($\hat K = \hat T$ or $\hat K = \hat Q$)
under diagonal scaling with matrix
$B = \bigl( \begin{smallmatrix} h_1 & 0 \\ 0 & h_2 \end{smallmatrix} \bigr)$, where $h_l > 0$.
Then for any $u \in H^{-1/2}(K)$ there holds
\be \label{H^{-1/2}-scaling}
    \|\hat u\|_{H^{-1/2}(\hat K)} \lesssim
		\frac{\max\,\{h_1^{1/2},\,h_2^{1/2}\}}{h_1 h_2} \|u\|_{H^{-1/2}(K)},
\ee
where $\hat\bu(\hat\bx) = \bu(B \hat\bx)$, $\hat\bx = (\hat x_1, \hat x_2) \in \hat K$.
\end{lemma}

\begin{proof}
By the definition of the dual norm
\be \label{H^{-1/2}-scaling_1}
    \|\hat u\|_{H^{-1/2}(\hat K)} =
		\stack{\hat v \in \tilde H^{1/2}(\hat K)}{\hbox{sup}}
		\frac{(\hat u,\hat v)_{0,\hat K}}{\|\hat v\|_{\tilde H^{1/2}(\hat K)}}.
\ee
We now estimate the norm $\|\hat v\|_{\tilde H^{1/2}(\hat K)}$.
If $\hat v \in H^1(\hat K)$, then diagonal scaling yields
\beas
     & \|\hat v\|^2_{0,\hat K} = (h_1 h_2)^{-1}\, \|v\|^2_{0,K}, &
		 \\[3pt]
		 & \|\hat\partial_1 \hat v\|^2_{0,\hat K} \simeq h_1 h_2^{-1}\, \|\partial_1 v\|^2_{0,K},\quad
		 \|\hat\partial_2 \hat v\|^2_{0,\hat K} \simeq h_1^{-1} h_2\, \|\partial_2 v\|^2_{0,K}. &
\eeas
Therefore,
\[
  \|\hat v\|^2_{H^1_0(\hat K)} =
	\|\hat\partial_1 \hat v\|^2_{0,\hat K} + \|\hat\partial_2 \hat v\|^2_{0,\hat K} \gtrsim
  \frac{\min\,\{h_1^2,h_2^2\}}{h_1 h_2}	\|v\|^2_{H^1_0(K)},
\]
and by interpolation between $L^2$ and $H^1_0$ we find that
\be \label{H^{-1/2}-scaling_2}
    \|\hat v\|^2_{\tilde H^{1/2}(\hat K)} \gtrsim
		\frac{\min\,\{h_1,h_2\}}{h_1 h_2}	\|v\|^2_{\tilde H^{1/2}(K)} =
		\frac{1}{\max\,\{h_1,h_2\}}	\|v\|^2_{\tilde H^{1/2}(K)}\quad
    \forall \hat v \in \tilde H^{1/2}(\hat K).
\ee
Since $(\hat u,\hat v)_{0,\hat K} = (h_1 h_2)^{-1}\,(u,v)_{0,K}$,
we use (\ref{H^{-1/2}-scaling_2}) in (\ref{H^{-1/2}-scaling_1}) to obtain
inequality (\ref{H^{-1/2}-scaling}).
\end{proof}

Now, we are in a position to prove the stability result and
the corresponding error estimate for the Raviart-Thomas interpolation
on anisotropic elements.

\begin{theorem} \label{thm_RT_diag_scaling}
Let $K$ be either the triangle $T$ with vertices
$(0,0)$, $(h_1,0)$, $(h_1,h_2)$,
or the rectangle $Q$ with vertices
$(0,0)$, $(h_1,0)$, $(0,h_2)$, $(h_1,h_2)$, where $h_l > 0$.
Denote $h_{\max} := \max\,\{h_1,h_2\}$.
Then for any $\bu \in \bH^{1/2}(K)$ with $\div\, \bu \in {\field{R}}$ there holds for $l=1,\,2$
\be \label{RT_diag_scaling_1}
    \|(\Pi_{\rm RT} \bu)_l\|^2_{0,K}
    \lesssim
    \|u_l\|^2_{0,K} +
    \frac{h_{\max}^3}{h_1 h_2}
    \left(
          |u_l|^2_{H^{1/2}(K)} + |u_{l+1}|^2_{AH_{l+1}^{1/2}(K)} +
          \|\div\, \bu\|^2_{H^{-1/2}(K)}
    \right)
    \ \hbox{\rm (mod. $2$)}
\ee
and
\be \label{RT_diag_scaling_2}
     \|u_l - (\Pi_{\rm RT} \bu)_l\|^2_{0,K}
     \lesssim
     \frac{h_{\max}^3}{h_1 h_2}
     \left(
           |u_l|^2_{H^{1/2}(K)} + |u_{l+1}|^2_{AH_{l+1}^{1/2}(K)} +
           \|\div\, \bu\|^2_{H^{-1/2}(K)}
     \right)
     \ \hbox{\rm (mod. $2$)}.
\ee
\end{theorem}

\begin{proof}
We consider only the case of the triangle, $K = T$.
The proof is similar in the case $K = Q$.

We use the Piola transformation to define
$\hat\bu \in \bH^{1/2}(\hat T) \cap \bH(\div,\hat T)$ on the reference triangle $\hat T$ as follows:
\[
  \hat\bu(\hat\bx) = h_1 h_2 B^{-1} \bu(B \hat\bx)\quad
  \hbox{with}\quad
  B = \begin{pmatrix} h_1 & 0 \\ 0 & h_2 \end{pmatrix}.
\]
Then we have 
\beas
     \|\hat u_1\|^2_{0,\hat T}
     & = &
     h_2^2\, (h_1 h_2)^{-1}\, \| u_1\|^2_{0,T} =
     h_2\, h_1^{-1}\, \|u_1\|^2_{0,T},
     \\[5pt]
     |\hat u_1|^2_{H^{1/2}(\hat T)}
     & = &
     \int_{\hat T} \int_{\hat T}
     \frac{|\hat u_1(\hat\bx) - \hat u_1(\hat\boldy)|^2}{(|\hat x_1 - \hat y_1|^2 + |\hat x_2 - \hat y_2|^2)^{3/2}}
     d\hat\bx\,d\hat\boldy
     \\[3pt]
     & = &
     h_2^2\, (h_1 h_2)^{-2}\,
     \int_{T} \int_{T}
     \frac{|u_1(\bx) - u_1(\boldy)|^2}
          {(h_1^{-2}|x_1 - y_1|^2 + h_2^{-2}|x_2 - y_2|^2)^{3/2}}
     d\bx\,d\boldy
     \\[3pt]
     & \le &
     h_1^{-2}\, h_{\max}^3\, |u_1|^2_{H^{1/2}(T)},
\eeas
and
\beas
     |\hat u_2|^2_{AH_2^{1/2}(\hat T)}
     & = &
     \int_0^1 |\hat u_2(\hat x_1,\cdot)|^2_{H^{1/2}(0,\hat x_1)}\,d\hat x_1
     \\[3pt]
     & = &
     h_1^{-1}\, \int_0^{h_1}
     \Bigg[
           \int_0^{h_2 x_1 / h_1} \int_0^{h_2 x_1 / h_1}
           \frac{h_1^2 |u_2(x_1,x_2) - u_2(x_1,y_2)|^2}
                {h_2^{-2}|x_2 - y_2|^2}\,
           \frac{dx_2\,dy_2}{h_2^2}
     \Bigg]
     dx_1
     \\[3pt]
     & = &
     h_1\, |u_2|^2_{AH_2^{1/2}(T)}.
\eeas
Furthermore, the standard properties of the Piola transformation yield
\[
  \Pi_{\rm RT}\bu(\bx) =
  \frac{1}{h_1 h_2} B \hat \Pi_{\rm RT} \hat\bu(\hat\bx) =
  \begin{pmatrix} 1/h_2 & 0 \\ 0 & 1/h_1 \end{pmatrix} \hat \Pi_{\rm RT} \hat\bu(\hat\bx)\quad\
  \hbox{and}\quad\
  \hat\div\, \hat\bu = {h_1 h_2}\, \div\, \bu \in {\field{R}}.
\]
Therefore, applying Theorem~\ref{thm_RT_triangle}~(ii)
and Lemma~\ref{lm_H^{-1/2}-scaling}
and using the fact that $\hat\div\, \hat\bu \in {\field{R}}$
(hence, $\|\hat\div\, \hat\bu\|_{0,\hat T} \simeq \|\hat\div\, \hat\bu\|_{H^{-1/2}(\hat T)}$), we obtain
\beas
     \|(\Pi_{\rm RT} \bu)_1\|^2_{0,T}
     & = &
     h_2^{-2} h_1 h_2\, \|(\hat \Pi_{\rm RT} \hat\bu)_1\|^2_{0,\hat T}
     \nonumber
     \\[3pt]
     & \lesssim &
     h_1 h_2^{-1}
     \left(
           \|\hat u_1\|^2_{0,\hat T} + |\hat u_1|^2_{H^{1/2}(\hat T)} +
           |\hat u_2|^2_{AH_2^{1/2}(\hat T)} + \|\hat\div\, \hat\bu\|^2_{0,\hat T}
     \right)
     \nonumber
     \\[3pt]
     & \lesssim &
     \|u_1\|^2_{0,T} +
     \frac{h_{\max}^3}{h_1 h_2} |u_1|^2_{H^{1/2}(T)}
     +\, \frac{h_1^2}{h_{2}} |u_{2}|^2_{AH_{2}^{1/2}(T)}
     + \frac{h_1}{h_2}\, h_{\max} \|\div\, \bu\|^2_{H^{-1/2}(T)}.
\eeas
Recalling that $h_{\max} = \max\,\{h_1,h_2\}$, it is easy to see that
\[
  \frac{h_1^2}{h_{2}} \le \frac{h_1}{h_2}\, h_{\max} \le \frac{h_{\max}^3}{h_1 h_2},
\]
and then inequality (\ref{RT_diag_scaling_1}) follows (for $l=1$ and $K = T$).

Arguing as above and using Corollary \ref{coroserge16/05:2}~(ii) instead of Theorem~\ref{thm_RT_triangle}~(ii)
we establish the error estimate in (\ref{RT_diag_scaling_2}) for $l=1$.
The proof is analogous in the case $l=2$.
\end{proof}

We can now estimate the $\bL^2$-error of the Raviart-Thomas interpolation
on the graded mesh $\Delta_h^\beta$ on $\G$.
The specific estimate that we need is for $\bH^{1/2}$-regular vector fields with discrete divergence.

\begin{lemma} \label{lem_RT-error}
For any $\bu \in \bH^{1/2}_{\perp}(\G)$ such that $\divg\bu \in \divg\bX_h$
there holds
\be \label{eq:RT-error}
    \|\bu - \Pi_{\rm RT}\bu\|_{0,\G} \lesssim
    h^{1-\beta/2}\, \|\bu\|_{\bH^{1/2}_{\perp}(\G)}.
\ee
\end{lemma}

\begin{proof}
Let $F$ be a face of $\G$,
and let $T_F \subset F$ be a triangular block of elements, see Figure~\ref{fig_1}
(the arguments are analogous for the parallelogram block of elements $Q_F$).
The triangle $T_F$ is mapped onto the unit triangle $\hat T$ by the affine transformation
which is independent of the mesh parameter $h$.
Let us first establish the error estimate for the Raviart-Thomas interpolation
on the unit triangle $\hat T$ partitioned into elements as shown in Figure~\ref{fig_1}.

The graded mesh on $\hat T$ comprises the quadrilaterals $K_{ij} = I_i \times I_j$
($i,j = 1,\ldots,N,\ i>j$) isomorphic to $(0,h_i) \times (0,h_j)$ with $h_i \ge h_j$
and the triangles $K_{ii}$ isomorphic to the triangle with vertices $(0,0)$, $(0,h_i)$, $(h_i,h_i)$.
Applying error estimates from Theorem~\ref{thm_RT_diag_scaling} on each element $K_{ij}$ ($i \ge j$),
we have for $l = 1,\,2$:
\[
  \|u_l - (\Pi_{\rm RT} \bu)_l\|^2_{0,K_{ij}}
  \lesssim
  {h_i^2 h_j^{-1}}
  \left(
        |u_l|^2_{H^{1/2}(K_{ij})} + |u_{l+1}|^2_{AH_{l+1}^{1/2}(K_{ij})} +
        \|\div\,\bu\|^2_{H^{-1/2}(K_{ij})}
  \right) \ \hbox{\rm (mod. $2$)}.
\]
Summing these estimates over all elements in $\hat T$ and recalling that
$h_i^2 h_j^{-1} \lesssim h^{2-\beta}$ for $1 \le i,j \le N$, we obtain
\be \label{eq:RT-error_aux1}
    \|\bu - \Pi_{\rm RT}\bu\|^2_{0,\hat T} \lesssim
    h^{2-\beta}\,
    \stack{i \ge j}{\sum\limits_{i,j=1}^{N}}
    \bigg(
          |\bu|^2_{\bH^{1/2}(K_{ij})} +
          \sum\limits_{l=1}^{2} |u_l|^2_{AH_l^{1/2}(K_{ij})} +
          \|\div\, \bu\|^2_{H^{-1/2}(K_{ij})}
    \bigg).
\ee
Note that by standard superposition argument
\bea
     \sum\limits_{i=1}^{N} \sum\limits_{j=1}^{i}
     |u_2|^2_{AH_2^{1/2}(K_{ij})}
     & \lesssim &
     \sum\limits_{i=1}^{N}
     \int_{I_i} |u_2(x_1,\cdot)|^2_{H^{1/2}(0,x_1)}\,dx_1
     \nonumber
     \\[3pt]
     & = &
     \int_0^1 |u_2(x_1,\cdot)|^2_{H^{1/2}(0,x_1)}\,dx_1 =
     |u_2|^2_{AH_2^{1/2}(\hat T)} \overset{\text{(\ref{serge9/05:11})}}{\lesssim}
     \|u_2\|^2_{H^{1/2}(\hat T)},\qquad\quad
     \label{eq:superposition}
\eea
and similarly for $u_1$.
Hence, using standard superadditivity properties of the $H^{1/2}$-semi\-norm
and the $H^{-1/2}$-norm, we deduce from~(\ref{eq:RT-error_aux1}) the following
error estimate on $\hat T$:
\[
    \|\bu - \Pi_{\rm RT}\bu\|_{0,\hat T} \lesssim
    h^{1-\beta/2}
    \left(\|\bu\|_{\bH^{1/2}(\hat T)} + \|\div\,\bu\|_{H^{-1/2}(\hat T)}\right).
\]
Applying now the Piola transformation associated with the mapping $T_F \to \hat T$,
patching together all individual blocks of elements on all faces of $\G$,
and using the superadditivity of $H^{1/2}$- and $H^{-1/2}$-norms (as the functions
of subdomains), we obtain
\be \label{eq:RT-error_aux2}
    \|\bu - \Pi_{\rm RT}\bu\|_{0,\G} \lesssim
    h^{1-\beta/2}
    \left(\|\bu\|_{\bH^{1/2}_{-}(\G)} + \|\divg\bu\|_{H^{-1/2}(\G)}\right)
\ee
(here and below we use the space $\bH^{s}_{-}(\G)$, $s>0$, which is defined
in a piecewise fashion by localisation to each face of $\G$,
with the norm $\|\bu\|^2_{\bH^{s}_{-}(\G)} := \sum_{F \subset \G} \|\bu|_{F}\|^2_{\bH^{s}(F)}$).

Inequality (\ref{eq:RT-error}) follows from~(\ref{eq:RT-error_aux2})
due to the continuity property of $\divg{:}\; \bH^{1/2}_{\perp}(\G) \to H^{-1/2}(\G)$
(see~\cite[Section~4.2]{BuffaC_01_TII}).
\end{proof}

\section{Discrete decomposition and the gap property} \label{sec_discrete_decomp}
\setcounter{equation}{0}

\rev{Following the ideas from~\cite{BuffaC_03_EFI} and~\cite[Section~9.1]{BuffaH_03_GBE}, we can use}
the Raviart-Thomas interpolation operator $\Pi_{\rm RT}$ to define \rev{the} discrete
counterparts of $\bV$ and $\bW$ in (\ref{decomp})
(e.g., we can set $\bV_h := \Pi_{\rm RT}(\sfR(\bX_h))$, where $\sfR$ is the regularised
projector introduced in Section~\ref{sec_decomp}).
However, as \rev{it} follows from the results in Section~\ref{sec_rt_stab},
the Raviart-Thomas interpolation of low-regular vector fields
on graded meshes $\Delta_h^{\beta}$ is only stable (with respect to the $\bL^2$-norm)
when $\beta < 2$.
Since the definition of the energy space $\bX$ for the EFIE involves
the dual space $\bH^{-1/2}_{\|}(\G)$ with a weaker norm than $\|\cdot\|_{0,\G}$,
we can relax the restriction on the grading parameter $\beta$
by employing a different projection onto the boundary element space
\rev{and using a duality argument on individual faces of $\G$.
This approach was successfully used
by Hiptmair and Schwab in~\cite[Section~8]{HiptmairS_02_NBE}
and by Buffa and Christiansen in~\cite[Section~4.2.2]{BuffaC_03_EFI}
in the context of the $h$-BEM with shape-regular meshes for the EFIE
(see~\cite{BespalovH_10_NpB, BespalovH_10_hpA} for applications of these
ideas to the analysis of the $p$-BEM and the $hp$-BEM with quasi-uniform meshes).
We will demonstrate below that using these techniques together with
stability properties and error estimates for the Raviart-Thomas interpolation
on anisotropic elements, one can design a stable discrete decomposition
of the boundary element space on $\Delta_h^{\beta}$ and prove the corresponding gap property~(\ref{gap})
for any $\beta < 3$.}

The construction of \rev{the desired projection operator} is technically involved.
Therefore, we formulate here the final result relevant to our discussion
and \rev{give a detailed proof in the next section}.
In the Proposition below, $\Pi_0$ denotes the $L^2(\G)$-projection onto the space of
piecewise constant functions over the mesh $\Delta_h^{\beta}$,
and $\bH^{-1/2}_{-}(\G)$ denotes the dual space of $\bH^{1/2}_{-}(\G)$
(with $\bL^2_{\rm t}(\G)$ as pivot space).

\begin{prop} \label{prop_Qh-projector}
There exists an operator $\CQ_h: \bH_{-}^s(\G) \cap \bH(\divg,\G) \to \bX_h$
($s > 0$) such that
\be \label{Qh_commut}
    \divg \circ \CQ_h = {\Pi_0} \circ \divg,
\ee
and for any $\eps > 0$
\be \label{Qh-error}
    \|\bu - \CQ_h\bu\|_{\bH^{-1/2}_{-}(\G)} \lesssim
    h^{1/2-\eps} \|\bu - \Pi_{\rm RT}\bu\|_{\bH(\divg,\G)}\quad
    \forall\,\bu \in \bH_{-}^s(\G) \cap \bH(\divg,\G).
\ee
\end{prop}

Thus, the operator $\CQ_h$ inherits the crucial
commuting diagram property (\ref{Qh_commut}) of the classical RT-interpolation operator
and, at the same time,
allows to gain an extra power of $h$ when estimating the error $(\bu - \CQ_h\bu)$
in the dual norm.

\begin{corollary} \label{cor_Qh-projector}
For any $\bu \in \bH^{1/2}_{\perp}(\G)$ such that $\divg\bu \in \divg\bX_h$, one has
$\CQ_h\bu \in \bX_h$,
$\divg\CQ_h\bu = \divg\bu$, and for any $\eps>0$ there holds
\be \label{Qh_error_X}
    \|\bu - \CQ_h\bu\|_{\bX} =
    \|\bu - \CQ_h\bu\|_{\bH^{-1/2}_{\|}(\G)} \lesssim
    h^{3/2-\beta/2-\eps} \|\bu\|_{\bH^{1/2}_{\perp}(\G)}.
\ee
\end{corollary}

Since $\|\cdot\|_{\bH^{-1/2}_{\|}(\G)} \lesssim \|\cdot\|_{\bH^{-1/2}_{-}(\G)}$,
this result immediately follows from Proposition~\ref{prop_Qh-projector} and
Lemma~\ref{lem_RT-error} due to the commuting diagram property for $\Pi_{\rm RT}$.

\medskip

Since $\sfR\bX_h \subset \bH^{1/2}_{\perp}(\G)$ (see~(\ref{R-bound})) and
$\divg\sfR\bX_h = \divg\,\bX_h$ (see~(\ref{div-property})),
the following definitions are valid thanks to Proposition~\ref{prop_Qh-projector}:
\[
  \bV_h := (\CQ_h \circ \sfR) \bX_h,\quad
  \bW_h := ({\rm Id} - \CQ_h \circ \sfR) \bX_h.
\]
Using the commuting diagram property (\ref{Qh_commut}), we have
\be \label{divg_Qh_R}
    \divg\CQ_h\sfR\bu_h = \Pi_0(\divg\sfR\bu_h)
    \overset{\text{(\ref{div-property})}}{=}
    \Pi_0(\divg\bu_h) = \divg\bu_h\quad
    \forall\,\bu_h \in \bX_h.
\ee
Therefore,
\be \label{RQhR}
    \sfR\CQ_h\sfR = \sfR\quad \hbox{on $\bX_h$},
\ee
and hence $\CQ_h \circ \sfR: \bX_h \to \bX_h$ is a projection.
This fact confirms that $\bX_h = \bV_h \oplus \bW_h$.
Property (\ref{divg_Qh_R}) also implies $\bW_h \subset \bW$,
and Corollary~\ref{cor_Qh-projector} yields stability of the discrete decomposition
in the following sense: there exists $C = C(\G,\beta)$ such that
\[
  \|\CQ_h\sfR\bu_h\|_{\bX} \le C\, \|\sfR\bu_h\|_{\bH^{1/2}_{\perp}(\G)}
  \overset{\text{(\ref{R-bound})}}{\le}
  C\, \|\divg\bu_h\|_{H^{-1/2}(\G)} \le C\,\|\bu_h\|_{\bX}\quad
  \forall\,\bu_h \in \bX_h,
\]
provided that $\beta < 3$.
This verifies property \textbf{(B)} from Section~\ref{sec_bem}.

It remains to establish the gap property \textbf{(C)}.
Inequality~(\ref{gap}) in~\textbf{(C)}
is another consequence of Corollary~\ref{cor_Qh-projector}:
\beas
     \sup_{\bv_h \in \bV_h} \inf_{\bv \in \bV}
     \frac{\|\bv - \bv_h\|_{\bX}}{\|\bv_h\|_{\bX}}
     & \le &
     \sup_{\bv_h \in \bV_h}
     \frac{\|\sfR\bv_h - \bv_h\|_{\bX}}{\|\bv_h\|_{\bX}}
     \overset{\text{(\ref{RQhR})}}{=}
     \sup_{\bv_h \in \bV_h}
     \frac{\|\sfR\bv_h - \CQ_h\sfR\bv_h\|_{\bX}}{\|\bv_h\|_{\bX}}
     \\[3pt]
     & \overset{\text{(\ref{Qh_error_X})}}{\lesssim} &
     h^{3/2-\beta/2-\eps}
     \sup_{\bv_h \in \bV_h} \frac{\|\sfR\bv_h\|_{\bH^{1/2}_{\perp}(\G)}}{\|\bv_h\|_{\bX}}
     \overset{\text{(\ref{R-bound})}}{\lesssim}
     h^{3/2-\beta/2-\eps}.
\eeas
This completes the proof of Theorem~\ref{thm_converge}.

\section{Proof of Proposition~\ref{prop_Qh-projector}} \label{sec_appendix}
\setcounter{equation}{0}

In this section, we give a constructive proof of Proposition~\ref{prop_Qh-projector}.
For any $\bu \in \bH_{-}^s(\G) \cap \bH(\divg,\G)$ we construct $\CQ_h\bu$
in the Raviart-Thomas spaces on individual faces of $\G$.
Let~$\tG$ be a single face of~$\G$, and let
$\Delta_h^\beta(\tG)$ denote the restriction of the graded mesh $\Delta_h^\beta$ onto~$\tG$.
For the sake of simplicity of notation we will omit the subscript $\tG$
for differential operators over this face, e.g., we will write $\div$ for $\div_{\tG}$.
We will also write $(\cdot,\cdot)$ for the $L^2(\tG)$- and $\bL^2(\tG)$-inner products, and
similarly $\|\cdot\|$ for the corresponding norms of scalar functions and vector fields.
First, let us prove the following auxiliary result.

\begin{lemma} \label{lm_aux_RT-stab}
For any $s > 1/2$, the Raviart-Thomas interpolation operator
$\Pi_{\rm RT}: \bH^s(\tG) \cap \bH(\div,\tG) \to \RT_0(\Delta_h^\beta(\tG))$,
is $\bL^2(\tG)$-stable, i.e., there exists a constant $C>0$ independent of $h$ such that
\be \label{aux_RT-stab}
    \|\Pi_{\rm RT} \bu\|_{0,\rkserge{\tG}} \le
    C \big(\|\bu\|_{\bH^s(\tG)} + \|\div\,\bu\|_{0,\rkserge{\tG}}\big) \quad
    \forall \bu \in \bH^s(\tG) \cap \bH(\div,\tG).
\ee
\end{lemma}

\begin{proof}
Similarly to the proof of Lemma~\ref{lem_RT-error}, it is sufficient to establish~(\ref{aux_RT-stab})
for the unit triangle $\hat T$ partitioned into elements as shown in Figure~\ref{fig_1}
(this is because the affine transformations that map triangular blocks of elements $T_{F}\,{\subset}\,\tG$
onto $\hat T$ are independent~of~$h$).

The graded mesh on $\hat T$ (see Figure~\ref{fig_1}) comprises anisotropic quadrilaterals $K_{ij} = I_i \times I_j$
($i,j = 1,\ldots,N,\ i>j$) isomorphic to $(0,h_i) \times (0,h_j)$ with $h_i \ge h_j$
and shape-regular triangles $K_{ii}$ isomorphic to the triangle with vertices $(0,0)$, $(0,h_i)$, $(h_i,h_i)$.
Using the Piola transform associated with the mapping $\hat K \to K_{ij}$
($i \ge j$; $\hat K = \hat T$ or $\hat Q$),
we define  $\hat\bu \in \bH^{s}(\hat K) \cap \bH(\div,\hat K)$.
Then, by the standard properties of the Piola transform we have
\[
     \|\hat u_1\|^2_{0,\hat K} \simeq
     h_i^{-1}\, h_j\, \|u_1\|^2_{0,K_{ij}}, \quad
     \|\hat\div\, \hat\bu\|^2_{0,\hat K} \simeq {h_i h_j}\,\|\div\,\bu\|^2_{0,K_{ij}},
\]
\[
     \|(\Pi_{\rm RT}\bu)_1\|^2_{0,K_{ij}} \simeq
     h_i\, h_j^{-1}\, \|(\hat \Pi_{\rm RT} \hat\bu)_1\|^2_{0,\hat K}.
\]
The application of the scaling argument yields:
\beas
     |\hat u_1|^2_{H^{s}(\hat T)}
     & = &
     \int_{\hat T} \int_{\hat T}
     \frac{|\hat u_1(\hat\bx) - \hat u_1(\hat\boldy)|^2}{(|\hat x_1 - \hat y_1|^2 + |\hat x_2 - \hat y_2|^2)^{1+s}}
     d\hat\bx\,d\hat\boldy
     \\[3pt]
     & \simeq &
     h_i^2\, (h_i h_i)^{-2}\,
     \int_{K_{ii}} \int_{K_{ii}}
     \frac{|u_1(\bx) - u_1(\boldy)|^2}
          {(h_i^{-2}|x_1 - y_1|^2 + h_i^{-2}|x_2 - y_2|^2)^{1+s}}
     d\bx\,d\boldy
     \\[3pt]
     & = &
     h_i^{2s}\, |u_1|^2_{H^{s}(K_{ii})},
\eeas
\beas
     |\hat u_1|^2_{AH_1^{s}(\hat Q)}
     & = &
     \int_0^1 |\hat u_1(\cdot,\hat x_2)|^2_{H^{s}(0,1)}\,d\hat x_2
     \\[3pt]
     & = &
     h_j^{-1}\, \int_0^{h_j}
     \Bigg[
           \int_0^{h_i} \int_0^{h_i}
           \frac{h_j^2 |u_1(x_1,x_2) - u_1(y_1,x_2)|^2}
                {h_i^{-1-2s}|x_1 - y_1|^{1+2s}}\,
           \frac{dx_1\,dy_1}{h_i^2}
     \Bigg]
     dx_2
     \\[3pt]
     & \simeq &
     h_i^{2s-1} h_j\, |u_1|^2_{AH_1^{s}(K_{ij})}\quad (i > j),
\eeas
and analogously,
\[
     |\hat u_1|^2_{AH_2^{s}(\hat Q)} \simeq
     h_i^{-1} h_j^{1+2s}\, |u_1|^2_{AH_2^{s}(K_{ij})}\quad (i > j).
\]  
Therefore, applying Theorem~\ref{thm_RT_triangle}~(i), we obtain
\bea
     \|(\Pi_{\rm RT} \bu)_1\|^2_{0,K_{ii}}
     & \simeq &
     \|(\hat \Pi_{\rm RT} \hat\bu)_1\|^2_{0,\hat T} \;\lesssim\;
     \|\hat u_1\|^2_{H^{s}(\hat T)} + \|\hat\div\,\hat\bu\|_{0,\hat T}
     \nonumber
     \\[3pt]
     & \simeq &
     \|u_1\|^2_{0,K_{ii}} + h_i^{2s}\,|u_1|^2_{H^{s}(K_{ii})} + h_i^{2}\,|\div\,\bu|^2_{0,K_{ii}}.
     \label{aux_RT_scaling_1}
\eea
Similarly, applying Theorem~\ref{thm_RT_square}~(i) and recalling (\ref{serge9/05:4}),
we have for $i>j$
\bea
     \|(\Pi_{\rm RT} \bu)_1\|^2_{0,K_{ij}}
     & \simeq &
     h_i\,h_j^{-1}\, \|(\hat \Pi_{\rm RT} \hat\bu)_1\|^2_{0,\hat Q} \;\lesssim\;
     h_i h_j^{-1}\, \|\hat u_1\|^2_{H^{s}(\hat Q)}
     \nonumber
     \\[3pt]
     & \simeq &
     h_i h_j^{-1}
     \left(
           \|\hat u_1\|^2_{0,\hat Q} + |\hat u_1|^2_{AH_1^{s}(\hat Q)} +
           |\hat u_1|^2_{AH_2^{s}(\hat Q)}
     \right)
     \nonumber
     \\[3pt]
     & \simeq &
     \|u_1\|^2_{0,K_{ij}} + h_i^{2s}\,|u_1|^2_{AH_1^{s}(K_{ij})} + h_j^{2s}\,|u_1|^2_{AH_2^{s}(K_{ij})}.
     \label{aux_RT_scaling_2}
\eea
The estimates analogous to (\ref{aux_RT_scaling_1}) and (\ref{aux_RT_scaling_2})
are also valid for $\|(\Pi_{\rm RT} \bu)_2\|_{0,K_{ij}}$ with $i \ge j$.

Combining the estimates for both components of  $\Pi_{\rm RT} \bu$ over all elements in $\hat T$ and then
using the superposition argument as in (\ref{eq:superposition}) for anisotropic seminorms and
the superadditivity property of the $H^{1/2}$-norm, we arrive at the desired result.
\end{proof}

\rev{Our construction of the operator $\CQ_h$ follows
the technique used by Hiptmair and Schwab
in the proof of Lemma~8.1 in~\cite{HiptmairS_02_NBE} but relies on stability properties
of the Raviart-Thomas interpolation on graded meshes over individual faces of $\G$.}
Given $\bu \in \bH^s(\tG) \cap \bH(\div,\tG)$, $s \,{>}\, 0$, we consider the
following mixed problem:
{\em Find $(\bz,f) \,{\in}\, \bH(\div,\tG) \,{\times}\, L^2_*(\tG)$ such~that}
\be \label{aux_problem}
    \ba{rll}
    (\bz,\bv) + (\div\,\bv,f)\,\, = & (\bu,\bv) & \quad \forall \bv \in \bH_0(\div,\tG),
    \\[2pt]
                  (\div\,\bz,g)\,\, = & (\div\,\bu,g) & \quad \forall g \in L^2_{*}(\tG),
    \\[2pt]
                   \bz \cdot \tbn\,\, = & \bu \cdot \tbn & \quad \hbox{on \ $\partial\tG$}.
    \ea
\ee
Here, $L^2_{*}(\tG) := \big\{v \in L^2(\tG);\; (v,1) = 0\big\}$,
$\tbn$ is the unit outward normal vector to $\partial\tG$,
\rev{and $\bH_0(\div,\tG) := \{\bv \in \bH(\div,\tG);\; \bv\cdot\tbn|_{\partial\tG} = 0\}$}.

The unique solvability of (\ref{aux_problem}) is proved by standard techniques
(see \cite[Chapter~II]{BrezziF_91_MHF}). In fact, it is clear that the pair $(\bu,0)$
solves (\ref{aux_problem}).

A conforming Galerkin approximation of problem (\ref{aux_problem}) with Raviart-Thomas elements
on the graded mesh $\Delta_h^\beta(\tG)$ reads as:
{\em Find $(\bz_h,f_h) \in \bX_h(\tG) \times R_{h}(\tG)$ such that}
\be \label{aux_Galerkin}
    \ba{rll}
    (\bz_h,\bv) + (\div\,\bv,f_h)\,\, = & (\bu,\bv) & \quad \forall \bv \in \bX_h(\tG) \cap \bH_0(\div,\tG),
    \\[2pt]
                    (\div\,\bz_h,g)\,\, = & (\div\,\bu,g) & \quad \forall g \in R_{h}(\tG),
    \\[2pt]
                   \bz_h \cdot \tbn\,\, = & \Pi_{\rm RT}\bu \cdot \tbn & \quad \hbox{on \ $\partial\tG$}.
    \ea
\ee
Here, $\bX_h(\tG)$ denotes the restriction of $\bX_h$ onto the face $\tG$, and
$R_{h}(\tG) := \{g \in L^2(\tG);\break
 g|_{K} = \hbox{\rm const},\ \forall\,K \in \Delta_h^\beta(\tG)\ \hbox{and}\ (g,1) = 0\}$.

\rev{Note that the third equation in~(\ref{aux_Galerkin}) implies $(\div(\bu-\bz_h),1) = 0$}.
Hence, the second identity in (\ref{aux_Galerkin}) holds for any
piecewise constant function $g \in \div\,\bX_h(\tG)$.
Thus, $\div\,\bz_h$ is the $L^2(\tG)$-projection of $\div\,\bu$ onto $\div\,\bX_h(\tG)$.
In particular, if $\div\,\bu \in \div\,\bX_h(\tG)$ then $\div\,\bz_h = \div\,\bu$.

We now prove the unique solvability of (\ref{aux_Galerkin}).
First, for any $g_h \in R_{h}(\tG)$ we find a function
$\phi \in H^1_{*}(\tG) := \{\phi \in H^1(\tG);\; (\phi,1) = 0\}$ solving the \rev{variational problem}
\be \label{Neumann}
    (\grad\phi,\grad\tilde\phi) = (g_h,\tilde\phi)\quad
    \forall\,\tilde\phi \in H^1_{*}(\tG).
\ee
Applying the standard regularity result for problem (\ref{Neumann})
(see, e.g.,~\cite[p.~82]{Grisvard_92_SBV}), we conclude that $\phi \in H^{1+r}(\tG)$
with some $r \in (\frac 12, \frac{\pi}{\omega})$
(here, $\omega < 2\pi$ denotes the maximal internal angle at the vertices of $\tG$), and
\be \label{grad_phi}
    \|\grad\,\phi\|_{\bH^r(\tG)} \lesssim \|\phi\|_{H^{1+r}(\tG)} \lesssim \|g_h\|.
\ee
Therefore, $\grad\,\phi \in \bH^r(\tG) \cap \bH_0(\div,\tG)$, $r > \frac 12$, and the interpolant
$\Pi_{\rm RT}\grad\,\phi \in \bX_h(\tG) \cap \bH_0(\div,\tG)$ is well defined and stable,
due to Lemma~\ref{lm_aux_RT-stab}.
Moreover,
$\div\,(\Pi_{\rm RT}\grad\,\phi) = \Pi_0(\div\,\grad\,\phi)
 \overset{\text{(\ref{Neumann})}}{=} g_h$.
Hence, using (\ref{aux_RT-stab}) and (\ref{grad_phi}) we prove
the discrete inf-sup condition:
\beas
     \stack{\bv_h \not= \bzero}{\sup_{\bv_h \in \bX_h(\tG) \cap \bH_0(\div,\tG)}}
     \frac{(\div\,\bv_h,g_h)}{\|\bv_h\|_{\bH(\div,\tG)}}
     & \ge &
     \frac{(\div\,(\Pi_{\rm RT}\grad\,\phi),g_h)}{\|\Pi_{\rm RT}\grad\,\phi\|_{\bH(\div,\tG)}}
     \\[2pt]
     & \ge &
     \frac{\|g_h\|^2}
          {C\left(\|\grad\,\phi\|_{\bH^r(\tG)} + \|\div\,\grad\,\phi\|\right) + \|\div\,(\Pi_{\rm RT}\grad\,\phi)\|}
     \\[2pt]
     & \ge &
     \tilde C\, \|g_h\|\quad \forall\,g_h \in R_{h}(\tG).
\eeas
This condition along with the property $\div\big(\bX_h(\tG) \cap \bH_0(\div,\tG)\big) = R_{h}(\tG)$
ensures existence, uniqueness, and quasi-optimality
of the solution $(\bz_h,f_h)$ to (\ref{aux_Galerkin}) (see \cite{BrezziF_91_MHF}).
In particular, using the quasi-optimality and recalling that $\bz = \bu$, $f=0$, we estimate
\bea
     \|\bu - \bz_h\|_{\bH(\div,\tG)}
     & \lesssim &
     \stack{(\bv_h - \Pi_{\rm RT}\bu)\cdot\tbn|_{\tG} = 0}{\inf_{\bv_h \in \bX_h(\tG)}}
     \|\bu - \bv_h\|_{\bH(\div,\tG)} +
     \inf_{g_h \in R_{h}(\tG)} \|f - g_h\|
     \nonumber
     \\[4pt]
     & \lesssim &
     \|\bu - \Pi_{\rm RT}\bu\|_{\bH(\div,\tG)}.
     \label{aux_u-zh}
\eea
We now estimate $\|\bu - \bz_h\|_{\tilde\bH^{-1/2}(\tG)}$.
One has for any $\eps \in (0,\frac 12)$
\be \label{smooth_1}
    \|\bu - \bz_h\|_{\tilde\bH^{-1/2}(\tG)} \le
    \|\bu - \bz_h\|_{\tilde\bH^{-1/2+\eps}(\tG)} =
    \sup_{\bw \in \bH^{1/2-\eps}(\tG)\setminus\{\bzero\}}
    \frac{\rev{|}(\bu - \bz_h,\bw)\rev{|}}{\|\bw\|_{\bH^{1/2-\eps}(\tG)}}.
\ee
\rev{For a given $\bw \in \bH^{1/2-\eps}(\tG)$, we solve the following problem}:
{\em Find $\varphi \in H^1_*(\tG)$ such that}
\be \label{smooth_2}
  (\grad\,\varphi,\grad\,\phi) =\rev{-(\bw, \grad \,\phi)}\quad
  \forall\,\phi \in H^1_*(\tG).
\ee
Similarly to (\ref{grad_phi}), the regularity result for $\varphi$ reads as
\be \label{smooth_4}
    \varphi \in H^{3/2-\eps}(\tG),\quad
    \|\varphi\|_{H^{3/2-\eps}(\tG)} \lesssim \rev{\|\tilde f\|_{(H^{1/2+\eps}(\tG))'}}
    \lesssim \|\bw\|_{\bH^{1/2-\eps}(\tG)},
\ee
\rev{where $\tilde f \in (H^{1/2+\eps}(\tG))'$ is defined by $\tilde f(\phi) = -(\bw, \grad \,\phi)$,
$\forall\, \phi \in H^{1/2+\eps}(\tG)$}.

Then we set
\be \label{smooth_5}
    \bq := \bw + \grad\,\varphi \in \bH^{1/2-\eps}(\tG) \cap \bH_0(\div,\tG).
\ee
It \rev{also} follows from (\ref{smooth_2}) that $\div\,\bq = \div\,\bw + \div\,\grad\,\varphi = 0$.
Furthermore, we have by (\ref{smooth_4})--(\ref{smooth_5}) that
\be \label{smooth_7}
    \|\bq\|_{\bH^{1/2-\eps}(\tG)} \lesssim \|\bw\|_{\bH^{1/2-\eps}(\tG)} + \|\varphi\|_{H^{3/2-\eps}(\tG)} \lesssim
    \|\bw\|_{\bH^{1/2-\eps}(\tG)}.
\ee
We now use (\ref{smooth_5}) and integration by parts to represent the numerator in (\ref{smooth_1}) as
\beas
     (\bu - \bz_h,\bw)
     & = &
     (\bu - \bz_h,\bq) - (\bu - \bz_h,\grad\,\varphi)
     \\[2pt]
     & = &
     (\bu - \bz_h,\bq) + (\div\,(\bu - \bz_h),\varphi) - ((\bu - \bz_h)\cdot\tbn,\varphi)_{0,\partial\tG}.
\eeas
Hence, using (\ref{aux_problem}), (\ref{aux_Galerkin}) and recalling that
$\bz = \bu$, $f = 0$, we find for any $\bq_h \in \bX_h(\tG) \cap \bH_0(\div,\tG)$
and arbitrary $\varphi_h \in R_h(\tG)$
\bea
     \rev{|}(\bu - \bz_h,\bw)\rev{|}
     & = &
     \rev{|}(\bu - \bz_h,\bq - \bq_h) + (\bu - \bz_h, \bq_h)
     \nonumber
     \\[2pt]
     &   &
     +\, (\div\,(\bu - \bz_h),\varphi - \varphi_h) - ((\bu - \Pi_{\rm RT}\bu)\cdot\tbn,\varphi)_{0,\partial\tG}\rev{|}
     \nonumber
     \\[3pt]
     & = &
     \rev{|}(\bu - \bz_h,\bq - \bq_h) + (\div\,\bq_h,f_h)
     \nonumber
     \\[2pt]
     &   &
     +\, (\div\,(\bu - \bz_h),\varphi - \varphi_h) - ((\bu - \Pi_{\rm RT}\bu)\cdot\tbn,\varphi)_{0,\partial\tG}\rev{|}
     \nonumber
     \\[3pt]
     & \le &
     \|\bu - \bz_h\|\,\|\bq - \bq_h\| + |(\div\,\bq_h,f_h)|
     + \|\div\,(\bu - \bz_h)\|\,\|\varphi - \varphi_h\|
     \nonumber
     \\[2pt]
     &   &
     +\, \|(\bu - \Pi_{\rm RT}\bu)\cdot\tbn\|_{H^{-1+\eps}(\partial\tG)}\,\|\varphi\|_{H^{1-\eps}(\partial\tG)}.
     \label{smooth_9}
\eea
Let $\Pi_{\rm RT}^{\rm q/u}$ denote the Raviart-Thomas interpolation operator on the `coarse' quasi-uniform
and shape-regular mesh $\Delta_h^{\rm q/u}(\tG)$ obtained from the graded mesh $\Delta_h^{\beta}(\tG)$
by patching together long and thin elements (see Figure~\ref{fig_3}).
We also denote by $\Pi_{0}^{\rm q/u}$ the $L^2(\tG)$-projector onto the space of
piecewise constant functions on $\Delta_h^{\rm q/u}(\tG)$.
Then we set
\[
  \bq_h := \Pi_{\rm RT}^{\rm q/u}\bq \in \bX_h(\tG) \cap \bH_0(\div,\tG)\quad
  \hbox{and}\quad
  \varphi_h := \Pi_{0}^{\rm q/u}\varphi \in R_h(\tG).
\]

\begin{figure}[!htb]
\begin{center}
\includegraphics[width=1\textwidth]{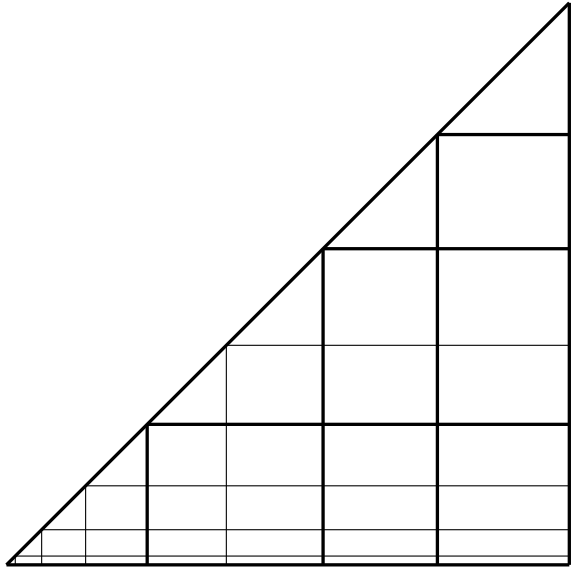}
\addtolength{\abovecaptionskip}{-2em}
\end{center}
\caption{An example of shape-regular quasi-uniform mesh (thicker lines) obtained by patching together elements
of the graded mesh (thinner lines).}
\label{fig_3}
\addtolength{\abovecaptionskip}{1em}
\end{figure}

By the standard properties of the Raviart-Thomas interpolation and the $L^2$-projection on quasi-uniform
and shape-regular meshes, we have
\be \label{div0}
    \div\,\bq_h = \Pi_{0}^{\rm q/u}\div\,\bq = 0,
\ee
\be \label{q-qh}
    \|\bq - \bq_h\| \lesssim h^{1/2-\eps}\, \|\bq\|_{H^{1/2-\eps}(\tG)}
    \overset{\text{(\ref{smooth_7})}}{\lesssim}
    h^{1/2-\eps}\, \|\bw\|_{H^{1/2-\eps}(\tG)},
\ee
\be \label{phi-phih}
    \|\varphi - \varphi_h\| \lesssim h\, \|\varphi\|_{H^{3/2-\eps}(\tG)}
    \overset{\text{(\ref{smooth_4})}}{\lesssim}
    h\, \|\bw\|_{H^{1/2-\eps}(\tG)}.
\ee
To estimate $\|(\bu - \Pi_{\rm RT}\bu)\cdot\tbn\|_{H^{-1+\eps}(\partial\tG)}$
we recall that $\int_{e_h} (\bu - \Pi_{\rm RT}\bu)\cdot\tbn = 0$ for any element edge
$e_h \subset \partial\tG$.
Therefore, we can use a standard duality argument to prove (cf.~\cite[p.~259]{BuffaC_03_EFI})
\[
  \|(\bu - \Pi_{\rm RT}\bu)\cdot\tbn\|_{H^{-1+\eps}(\partial\tG)} \lesssim
  \Big(\max_{e_h \subset \partial\tG} |e_h|\Big)^{1-\eps}
  \|(\bu - \Pi_{\rm RT}\bu)\cdot\tbn\|_{L^2(\partial\tG)}.
\]
Then by interpolation we obtain
\be \label{boundary_est}
    \|(\bu - \Pi_{\rm RT}\bu)\cdot\tbn\|_{H^{-1+\eps}(\partial\tG)} \lesssim
    h^{1/2-\eps}\,\|(\bu - \Pi_{\rm RT}\bu)\cdot\tbn\|_{H^{-1/2}(\partial\tG)} \lesssim
    h^{1/2-\eps}\,\|\bu - \Pi_{\rm RT}\bu\|_{\bH(\div,\tG)}
\ee
(here, we also used the continuity of the normal trace operator
$\bv \mapsto \bv\cdot\tbn|_{\partial\tG}$ as a mapping
$\bH(\div,\tG) \to H^{-1/2}(\partial\tG)$).

Furthermore, one has
\be \label{phi_est}
    \|\varphi\|_{H^{1-\eps}(\partial\tG)} \lesssim \|\varphi\|_{H^{3/2-\eps}(\tG)}
    \overset{\text{(\ref{smooth_4})}}{\lesssim}
    \|\bw\|_{H^{1/2-\eps}(\tG)}.
\ee
Now, using (\ref{div0})--(\ref{phi_est}) in (\ref{smooth_9}) and
recalling (\ref{aux_u-zh}) we find
\[
  \rev{|}(\bu - \bz_h,\bw)\rev{|} \lesssim
  h^{1/2-\eps}\,\|\bu - \Pi_{\rm RT}\bu\|_{\bH(\div,\tG)}\,\|\bw\|_{H^{1/2-\eps}(\tG)}.
\]
Using this estimate in (\ref{smooth_1}) we obtain
\be \label{u-zh_final}
    \|\bu - \bz_h\|_{\tilde\bH^{-1/2}(\tG)} \lesssim
    h^{1/2-\eps}\,\|\bu - \Pi_{\rm RT}\bu\|_{\bH(\div,\tG)}.
\ee

Now we can prove the desired result.

\medskip

\noindent
{\bf Proof of Proposition~\ref{prop_Qh-projector}.} 
For any $\bu \in \bH^s_{-}(\G) \cap \bH(\divg,\G)$, we define $\CQ_h\bu \in \bX_h$
face by face as $\CQ_h\bu|_{\tG} := \bz_h$ for any face $\tG \subset \G$,
where $\bz_h$ is a unique (vectorial) solution to~(\ref{aux_Galerkin}).
Then the commuting diagram property (\ref{Qh_commut}) follows from the second
identity in~(\ref{aux_Galerkin}), and inequality (\ref{u-zh_final}) yields estimate (\ref{Qh-error}).
\qed

\noindent{\bf Acknowledgement.}
A significant part of this work has been done while A.B. was visiting
Laboratoire de Math\'ematiques et ses Applications de Valenciennes,
Universit\'e de Valenciennes et du Hainaut-Cambr\'esis (Valenciennes, France).
This author is grateful to the colleagues in that department
for their hospitality and stimulating research atmosphere.

\bibliographystyle{siam}
\bibliography{../../../../../bibtex/bib}

\end{document}